\documentclass[12pt]{amsart}

\usepackage{fancybox}
\newcommand{\comment}[1]{\shadowbox{\footnotesize #1}}
\newcommand{\longcomment}[1]{\fbox{\begin{minipage}{\textwidth}\footnotesize
#1\end{minipage}}}
\renewcommand{\comment}[1]{}
\renewcommand{\longcomment}[1]{}

\usepackage{amssymb,amsthm,amsfonts,dsfont}
\usepackage{amsmath}

\def\<{\langle}\def\>{\rangle}

 \def\E{{\mathds{E}}}

 \def\tr{\, \textrm{\rm tr}\,}

\newcommand{\PP}[1]{\mathbb{P}_{#1}}

 \newtheorem{theorem}{Theorem}[section]
       \newtheorem{proposition}[theorem]{Proposition}
       \newtheorem{corollary}[theorem]{Corollary}
       \newtheorem{lemma}[theorem]{Lemma}
           \newtheorem{claim}[theorem]{Claim}
       \newtheorem{theoremA}{Proposition}

\theoremstyle{definition}
\newtheorem{definition}{Definition}[section]

       \newtheorem{remark}{Remark}[section]

\newtheorem{question}{ }[subsection]
\newcommand{\la}{\lambda}
\newcommand{\La}{\Lambda}

\newcommand{\RR}{\mathds{R}}
\newcommand{\CC}{\mathds{C}}

\newcommand{\NN}{\mathds{N}}
\newcommand{\mU}{{\bf U}}

\newcommand{\mX}{\mathbf{X}}
\newcommand{\mY}{\mathbf{Y}}
\newcommand{\mZ}{\mathbf{Z}}
\newcommand{\mS}{\mathbf{S}}
\newcommand{\mI}{\mathbf{I}}
\newcommand{\mR}{\mathbf{R}}
\newcommand{\mM}{{\rm M}}
\newcommand{\mO}{\mathbf{0}}
\numberwithin{equation}{section}

\newcommand{\nX}{\widetilde{\mX}}
\newcommand{\nY}{\widetilde{\mY}}

\newcommand{\nS}{\widetilde{\mS}}

\newcommand{\calI}{\mathcal{I}}
\newcommand{\calJ}{\mathcal{J}}
\newcommand{\calL}{\mathcal{L}}
\newcommand{\calK}{\mathcal{K}}
\newcommand{\HH}{\mathds{H}}

\newcommand{\mP}{\mathbf{P}}

\newcommand{\Var}{{\rm{Var}}}

\newcommand{\macierz}[1]{\left[\begin{matrix}#1\end{matrix}\right]}

\title{Meixner  matrix  ensembles}
\author{ W{\l}odzimierz Bryc} \thanks{Department of Mathematical Sciences,
University of Cincinnati, Cincinnati, OH 45221-0025, USA}
\author{ G\'erard Letac}\thanks{Laboratoire de Statistique et
Probabilit\'es, Universit\'e Paul Sabatier, 31062, Toulouse, France}

 \date{Created: June 2006. Revised: Sept 2010, and March 2011. Printed: \today \ file: \jobname.tex}

\newcommand{\JJJJ}{\mathbb{J}}

\begin{document}

 \maketitle
 \begin{abstract} We construct a family of matrix ensembles  that fits  Anshelevich's regression
 postulates for  ``Meixner laws on matrices", namely the distribution with an invariance property of $\mX$ when $\E(\mX^2|\mX+\mY)=a(\mX+\mY)^2+b(\mX+\mY)+c\mI_n$ where $\mX$ and $\mY$ are i.i.d. on symmetric matrices of order $n$.
We show that   the Laplace transform of a general  $n\times n$ Meixner matrix ensemble satisfies a
system of partial differential equations which is explicitly solvable for $n=2$.  We rely on these solutions to identify the
six types of    $2\times 2$ Meixner matrix ensembles.
 \end{abstract}

\section{Introduction}
The  classical  Meixner laws are the one dimensional binomial, Poisson, negative binomial, gamma,  normal, and
hyperbolic  Meixner laws that include the %
hyperbolic secant law. These laws make their appearance  as the
orthogonality measures of a certain family of orthogonal  polynomials \cite{Meixner-40}, as the laws characterized
by quadratic regression property \cite{Laha-Lukacs60}, and as the exponential families corresponding to the
quadratic variance functions \cite{Morris82}, see also  \cite{Feinsilver:1986,Ismail-May78,Letac:1989,Shanbhag:1972,Shanbhag:1979}. The Laha-Lukacs
\cite{Kagan-Linnik-Rao:1973,Laha-Lukacs60} characterization is the description of the possible distributions of a
square-integrable   random variable $X$ such that there exist real numbers $a,b,c$ satisfying
$\E(X^2|S)=aS^2+bS+c$ where $S=X+Y,$ where $Y\sim X$ and where $Y$ is independent of $X.$ Another convenient form
of this condition is to postulate the existence of real numbers $A,B,C$ such that
\begin{equation}\label{BT}\E((X-Y)^2|S)=AS^2+BS+C\end{equation} (use $(X-Y)^2=2X^2+2Y^2-S^2$ to see that
$(A,B,C)=(4a-1,4b,4c)$).

In 2006 M. Anshelevich \cite[pages 22 and 25]{Anshelevich-06}  observes that (\ref{BT}) makes sense when the real
random variable $X$ is replaced by a random square  matrix $\mX$. He then raises the question of defining
analogues of Meixner distributions on matrices. In other words, he asks for the matrix version of Laha-Lukacs
\cite{Laha-Lukacs60} result. This question on random matrices is so general that it is suitable to  restrict this
problem to random variables $\mX$ valued in the space $\HH_{n,1}$ of all symmetric matrices and  such that %
$U\mX U^*\sim \mX$ for  all matrices $U$ in the orthogonal group $\mathbb{O}(n)=\mathcal{K}_{n,1}.$  On the other
hand, it is natural to extend this simpler framework to the space $\HH_{n,2}$ of Hermitian complex matrices for
random variables $\mX$ such that $U\mX U^*\sim \mX$ for  all matrices $U$ in the unitary  group
$\mathbb{U}(n)=\mathcal{K}_{n,2} $ and even  to the space $\HH_{n,4}$ of Hermitian-quaternionic matrices, again
for random variables $\mX$ such that $U\mX U^*\sim \mX$ for  matrices $U$ in the symplectic group
$\mathbb{S}p(n)=\mathcal{K}_{n,4} $.  There are also connections with free probability (see \cite{Bozejko-Bryc-04} and Section \ref{Sect:CWFP} below), and connections with Jordan
algebras (see Section \ref{Sect7} below).

A probability law $\mu$ on $\HH_{n,\beta}$ (where $\beta=1,2$ or 4), or a random variable $\mX$ with law $\mu$, is %
 said to be rotation invariant if  the law of $U\mX U^*$ is also $\mu$ for all $U$ in $\mathcal{K}_{n,\beta}$. In such case, following a physicists tradition we shall call $\mu$ an ensemble.

\begin{definition}\label{Def1.1} We will say that the probability law $\mu$ on $\HH_{n,\beta}$ is a Meixner ensemble  with parameters $A,B,C\in\RR$ if $\mu$ is  rotation invariant, the second moments exist, and $\mu$ has the following property:
if $\mX,\mY$ are independent with the same law $\mu$ and   $\mS=\mX+\mY$,
then
\begin{equation}\label{pre-ME}
\E ((\mX-\mY)^2|\mS)=A \mS^2 + B\mS+C\mI_n\,.
\end{equation}
\end{definition}
This is a different concept of a Meixner matrix ensemble than the one  introduced in
\cite{chikuse1986multivariate} under the name  ``multivariate Meixner classes of invariant distributions of random
matrices".
In particular, we note that the oldest known ensemble, namely the Wishart distribution with mean proportional to $\mI_n$ is \textit{not}  a Meixner ensemble in the  sense of Definition \ref{Def1.1}, since if $\mX,\mY$ are i.i.d. with Wishart distribution, then there are constants $A,B$  such that
\begin{equation}\label{WW}
\E ((\mX-\mY)^2|\mS)= A\mS^2+B\mS\tr\mS,
\end{equation}
which is quadratic in $\mS$ but not of  the desired form (\ref{pre-ME}). For a proof of (\ref{WW}), see \cite[page 582]{Letac-Massam-98}.

To complete this introduction, let us make a number of simple remarks about the Meixner ensembles. It is easy to
check that if the law of  $\mX$ is  a Meixner ensemble with parameters $A,B,C$, then an affine transformation
$\widetilde \mX=\alpha\mX+t$ with real $\alpha,t$ is also Meixner, with parameters
$$
\widetilde A=A, \; \widetilde B=\alpha B-4At,\; \widetilde C=\alpha^2C+4At^2-2B\alpha t.
$$
In particular, since passing to $-\mX$ changes only the sign of $B$, without loss of generality we will consider
only $B\geq 0$. Suppose now that $\E(\mX)$ and  $\E(\mX^2)$ are given; in this case  the constants $A,B,C$ satisfy
an equation with coefficients that depend on  $\E(\mX)$, $\E(\mX^2)$. The equation is obtained from taking the
expected values of both sides of \eqref{pre-ME}. Finally, rotation invariance implies that both $\E(\mX)$ and
$\E(\mX^2)$ are constant multiples of $\mI_n$, so  in the non-degenerate case $\E(\mX^2)\neq 0$ we can normalize
the matrices so that $\E (\nX)=0$ and $\E (\nX^2)=\mI_n$, and express  relation
 \eqref{pre-ME}  in terms of two
real constants $a,b$  as
\begin{equation}
  \label{Meixner-Anshelevich}
  \E\big((\nX-\nY)^2\big|\nS\big)=\frac{2}{1+2a}\left(\mI_n +b\nS+a \nS^2\right)\,.
\end{equation}
In this normalized setting one can classify
 the Meixner ensembles into
\begin{enumerate}
  \item the elliptic case $b^2>4a$, which in Laha-Lukacs theorem on $\RR$ corresponds to negative binomial ($a>0$), Poisson ($a=0$), or binomial ($a<0$) laws.
  \item the parabolic case $b^2=4a$, which corresponds to gamma law ($a>0$) or Gaussian law ($a=0$)
  \item the hyperbolic case $b^2<4a$, which corresponds to the family of hyperbolic Meixner laws that include hyperbolic secant law.
\end{enumerate}
Of course,  for non-degenerate ensembles \eqref{pre-ME} and \eqref{Meixner-Anshelevich} are equivalent. If
\eqref{pre-ME} holds with constants $A\ne 1,B,C$, then the standardization $\nX$ of $\mX$ with mean $\mu\mI_n$ and
variance $\Var(\mX):=\E(\mX^2)-(\E(\mX))^2=\sigma^2\mI_n$ satisfies \eqref{Meixner-Anshelevich} with
\begin{equation}\label{ABC2ab}
a=\frac{A}{2(1-A)},\; b=\frac{B+4 A\mu}{2 \sigma(1-A)}.
\end{equation}

What we call  trivial Meixner ensembles are  random multiples of the identity, $\mX=\xi \mI_n$ where $\xi$ is
$\RR$-valued with one of the six Meixner laws. The aim of the paper is the study of non-trivial Meixner ensembles
for all $n$. We will be able to find all of them for $n=2$, we will give examples of Meixner ensembles for $n\geq
3$ and we will write  a system of $n$ linear partial  differential equations (PDEs)  of the second order satisfied by the
Laplace transform of a Meixner ensemble of given parameters $(a,b)$ as defined by \eqref{Meixner-Anshelevich}. Let
us emphasize the fact that several different ensembles correspond to one set of parameters $a,b$ in
\eqref{Meixner-Anshelevich}.

The zoo of the known Meixner ensembles includes obviously the Gaussian ensembles which are described by their cumulant transform $c_1\tr \theta+\frac{c_2}{2}(\tr \theta)^2+\frac{c_3}{2}\tr (\theta^2)$ with the familiar GOE, GUE and GSE (see
\cite{Anderson-Guionnet-Zeitouni:2010} or \cite{Mehta91}) corresponding to $c_3=1$ and $c_1=c_2=0$  among them.   The Gaussian ensembles    yield a family of Meixner
ensembles corresponding to $a=b=0$ in \eqref{Meixner-Anshelevich}, but the existence of non Gaussian ensembles for $a=b=0$ and $n\geq 3$ has not been proved or disproved. Another important example is a Bernoulli ensemble: if $\mu_m$ is the distribution of the projection on a
uniformly random $m$ dimensional space  of the Euclidean space of dimension $n$, a Bernoulli ensemble is any
mixing of $\mu_0,\ldots,\mu_n$; for $n=1$ we get back the ordinary Bernoulli distribution. The Bernoulli ensembles
can be seen as the only distributions of random projections $\mP$ such that $U\mP U^*\sim \mP$ for all $U\in
\mathcal{K}_{n,\beta}.$  The convolution of $N$ identical Bernoulli ensembles provides the binomial ensemble, and
randomizing $N$ by Poisson or negative binomial distribution yields the Poisson and the negative binomial
ensembles. The construction of the full family of Meixner ensembles with $b^2=4a$ or with $b^2<4a$ is available only for $n=2$ and their extension
to $n\geq 3$ is  a challenge.

 The paper is organized as follows.  In Section \ref{Sect_Laplace} we describe Laplace transforms of Meixner
ensembles. In Section \ref{Sect:Meixner_Ensemble} we give examples of Meixner ensembles.    In Section
\ref{Sect:PDE} we derive the system of  PDEs  that determine the Laplace transforms of  the general $n\times n$
Meixner ensembles. The general solution of this system is elusive for $n\geq 3$.   In  Section
\ref{Morris-Meixner} we use the system of PDEs to show  that under a natural integrability assumption the examples
of Meixner ensembles from Section \ref{Sect:Meixner_Ensemble} exhaust all $2\times 2$ Meixner ensembles. In
Section \ref{Sect_Add_Obs} we collected  additional material: Proposition \ref{C2.3}  answers in negative a
question raised  in  \cite[Remark 5.8]{Bozejko-Bryc-04} for $2\times 2$ random matrices with finite Laplace
transform. Proposition \ref{P-P1} gives a series expansion for a Laplace transform of a random projection which
can be used to derive (some)   solutions to the system of PDEs. Section \ref{Sect7} makes some comments about the
Jordan algebra context of the problem.

\section{Laplace transforms}\label{Sect_Laplace}

In this section we    rewrite the %
Meixner property \eqref{pre-ME} in terms of the
Laplace transform.
We will work simultaneously with symmetric, unitary and symplectic ensembles, thus with probability laws on  $\HH_{n,\beta}$ which are invariant under the action $X\mapsto UXU^*$ where $U$ is in the group $\mathcal{K}_{n,\beta}.$
We call  $\beta=1,2,4$ the Peirce constant    (see \cite[page 97]{Faraut-Koranyi-94}, where $\beta$ is denoted as $d$  and $n$ is denoted by $r$). The space
$\HH_{n,\beta}$ is
 a real vector space equipped with the Euclidean
structure defined by the real inner product
$$ (x,y)\mapsto \langle x|y\rangle:=\Re(\tr(xy))$$
 (the real part of the trace is needed  for the symplectic case $\beta=4$ only).
For an $\HH_{n,\beta}$-valued random variable $\mX$, the Laplace transform
  is a mapping from  $\HH_{n,\beta}$ into $(0,\infty]$, defined by
$$L(\theta)=\E (\exp(\langle\theta|\mX\rangle)).$$
Let
\begin{equation}\label{Theta_mu}
\Theta_\mX =\mbox{int}\{\theta: L(\theta)<\infty\}.
\end{equation}
(We will also use notation $\Theta_\mu$ when $\mu$ is the law of random variable $\mX$.)
Throughout this paper we consider  only  random matrices $\mX$ with values in $\HH_{n,\beta}$ with non-empty
$\Theta_\mX$.

 We need a result of linear
algebra that we are not going to prove. Consider
the space $L_s(\HH_{n,\beta})$ of symmetric endomorphisms of the Euclidean space $\HH_{n,\beta}.$
To each $y\in \HH_{n,\beta}$ we associate the elements $y\otimes y$ and
$\PP{y}$ of $L_s(\HH_{n,\beta})$ defined respectively by
$$h\mapsto (y\otimes y)(h)=y\tr (yh),\ \ h\mapsto \PP{y}(h)=yhy\,.$$
They provide important examples of elements of $L_s(\HH_{n,\beta})$.
 Now, $L_s(\HH_{n,\beta})$ is itself a linear space, and the result
that we are going to admit as a black box is the following (see
\cite[Lemma 6.3]{Casalis-Letac-96} and \cite[Proposition
3.1]{Letac-Massam-98}
for a proof):

\begin{theoremA}
\label{P1.1}
 There exists a unique endomorphism $\Psi$
of $L_s(\HH_{n,\beta})$ such that for all $y\in \HH_{n,\beta}$ one
has $\Psi(y\otimes y)=\PP{y}.$ Furthermore
\begin{equation}\label{magie}\Psi(\PP{y})=\frac{\beta}{2}y\otimes
y+\left(1-\frac{\beta}{2}\right)\PP{y}\,.\end{equation}
\end{theoremA}

 With this notation we have the
simple but  crucial result. 
(Note that in this result we could have dispensed with the  hypothesis of invariance  under rotations.)

\begin{proposition} \label{P1.2}
  Let $\mu$ be an ensemble on  $\HH_{n,\beta}$ such that its Laplace transform $L$ is finite on an open non empty set. The following are equivalent:
  \begin{enumerate}
\item The ensemble $\mu$  is  Meixner with real parameters $A,B,C$, {\em i.e.} \eqref{pre-ME} holds for $\textbf{X}\sim \mu.$
 \item
 With $k(\theta)=\ln L(\theta)$, we have
 \begin{equation}\label{pre-k-eqtn}
 2(1-A)\Psi(k''(\theta))(\mI_n)=4A (k'(\theta))^2+2Bk'(\theta)+C\mI_n
\end{equation}
for all $\theta\in\Theta_X$.
\end{enumerate}
\end{proposition}

\begin{proof}
To prove (i)$\Rightarrow$(ii), we first observe that
\begin{equation}\label{First Obs}
\E \left((\mX-\mY)\otimes(\mX-\mY)e^{\langle \theta|\mS\rangle}\right)=2L''(\theta)L(\theta)-2L'(\theta)\otimes L'(\theta).\end{equation}
To see  this, note that since $\mX$, $\mY$ are i.i.d, the left hand side of (\ref{First Obs}) is
$$
2L(\theta)\E \left(\mX \otimes\mX e^{\langle \theta|\mX\rangle}\right)-2 \E \left(\mX e^{\langle \theta|\mX\rangle}\right)\otimes \E \left(\mY e^{\langle \theta|\mY\rangle}\right).
$$
The same reasoning gives
\begin{equation}\label{Second Obs}
\E \left(\mS\otimes\mS e^{\langle \theta|\mS\rangle}\right)=2L''(\theta)L(\theta)+2L'(\theta)\otimes L'(\theta).\end{equation}
 Since $\E \left(\mS  e^{\langle \theta|\mS\rangle}\right)=2L(\theta)L'(\theta)$, applying $\Psi$ to \eqref{First
Obs} and \eqref{Second Obs}, from
  \eqref{pre-ME}   we get %
  \begin{equation}\label{pre-L-eq}
  2 (1-A)\Psi(L''(\theta))(\mI_n)L(\theta)=2(1+A)(L'(\theta))^2+2B L(\theta)L'(\theta)+C(L(\theta))^2\mI_n.
\end{equation}
  Since $L'=k' L$ and $L''=(k''+k'\otimes k')L$, this implies \eqref{pre-k-eqtn}.

Implication (ii)$\Rightarrow$(i) also follows from the above identities. Applying again $\Psi$ to \eqref{First Obs} and  \eqref{Second Obs}, from \eqref{pre-L-eq} we infer that
$$
\E\left( (\mX-\mY)^2 e^{\langle \theta|\mS\rangle}\right)=\E\left((A \mS^2+B\mS+C\mI) e^{\langle \theta|\mS\rangle}\right)
$$
for all $\theta\in \Theta_\mX\ne\emptyset$.
It is  known that this is equivalent to \eqref{pre-ME},  see \cite[Section 1.1.3]{Kagan-Linnik-Rao:1973}.

\end{proof}
\begin{remark}
 In the above proposition, the use of the endomorphism $\Psi$ may seem surprising. For a random variable $\mX$ valued
in a linear space $E$, the link between $\E(\mX\otimes \mX)$  and the Laplace transform $L$ of $\mX$ is easy
through the second differential of $L$. However, if $E$ is an algebra, like the space of square real matrices,
relating $\E(\mX^2)$  to $L$ is difficult. Here our reasoning was based on $\PP{\mX}(A)=\mX A\mX$ and thus
$\PP{\mX}(\mI_n)=\mX^2.$ Some other facts about $\Psi$ are known; it is a symmetric operator on
$L_s(\HH_{n,\beta})$ with only two eigenspaces: the one generated by all the $\frac{\beta}{2}y\otimes y+\mathbb{P}_y$ for the eigenvalue 1 and the
one generated by all the $y\otimes y-\mathbb{P}_y$ %
for the eigenvalue $-\beta/2.$ Details are in \cite[Section 5]{Letac-Wesolowski:2010} where the dimensions of these two subspaces are also computed.
\end{remark}
\begin{remark}\label{Rem:Exp-Fam} Let $\mu$ be a Meixner ensemble with cumulant function $k$ that satisfies   \eqref{pre-k-eqtn}. If $\theta_0\in\Theta_\mu$, then $k_{\theta_0}(\theta)=k(\theta+\theta_0)-k(\theta_0)$ is the cumulant function of the probability measure
$$
\mu_{\theta_0}(dx)=e^{\<\theta_0|x\>-k(\theta_0)}\mu(dx).
$$
In other terms, $\mu_{\theta_0}$ is a member of the natural exponential family generated by $\mu$. Clearly, $k_{\theta_0}$ also satisfies \eqref{pre-k-eqtn}.
However, $\mu_{\theta_0}$ is an ensemble {\em i.e.} it is rotation invariant, if and only if $\theta_0=t\mI_n$ for some real $t$. In this case $\<\theta_0|x\>=t\tr x$.
\end{remark}

\begin{remark} Note that  if $A=1$ then $\mX$ is degenerate.
Indeed, from \eqref{pre-k-eqtn} we see that $A=1$ implies $k'(\theta)=\mbox{const}$.
\end{remark}

We now show that Meixner ensembles are preserved under convolution power.
\begin{definition}\label{Def:Jorgensen}
Let $\mu$ be a probability measure  on a finite dimensional linear space  such that its Laplace transform is finite on a non empty open set.
The J\o rgensen set $\La(\mu)$ of  $\mu$  is the set of $\alpha>0$
such that there exists a probability measure $\mu_\alpha$ with   $L_{\mu_\alpha}=L_\mu^\alpha$ and $\Theta_{\mu_\alpha}=\Theta_\mu$, see e.g. \cite[page 767]{Casalis-Letac-96}.
\end{definition}
For instance, $\La(\mu)=(0,\infty)$ if $\mu$ is infinitely divisible, and  $\La(\mu)$ is the set of positive integers if $\mu$ is the  Bernoulli distribution on $\{0,1\}.$
\begin{proposition}\label{P2.2} Suppose $\Theta_\mu\ne \emptyset$.
If $\mu$ is  a Meixner ensemble with parameters $A\ne 1,B,C$ and if $\alpha\in\La(\mu)$, then $\mu_\alpha$ is  a Meixner ensemble with parameters
$$
A_\alpha=A/(A+\alpha(1-A)),\; B_\alpha=\alpha B/(A+\alpha(1-A)),\; C_\alpha=\alpha^2C/(A+\alpha(1-A)).
$$
\end{proposition}

\begin{proof} %
 It is clear that $\mu_\alpha$ is invariant under rotations. Since $k_\alpha(\theta):=k_{\mu_\alpha}(\theta)=\alpha k_\mu(\theta)$,
multiplying \eqref{pre-k-eqtn} by $\alpha(1-A_\alpha)/(1-A)$ we get
$$2(1 -A_\alpha)\Psi(k''_\alpha(\theta))(\mI_n)=4\frac{A (1-A_\alpha)}{\alpha(1-A)} \left(k'_\alpha(\theta)\right)^2+
\frac{2B(1-A_\alpha)}{1-A} k'_\alpha(\theta)+\frac{C\alpha(1-A_\alpha)}{1-A}.
$$
Solving the equation  $$A_\alpha=\frac{A (1-A_\alpha)}{\alpha(1-A)},$$
we get the formulas for $A_\alpha,B_\alpha,C_\alpha$ such that \eqref{pre-k-eqtn} holds.
\end{proof}

It will be convenient to have a version of Proposition \ref{P1.2}  for standardized ensembles.
\begin{corollary}\label{C-P1.2}
Suppose  $\E (\nX)=\mO$, $\E (\nX^2)=\mI_n$.
Denote $k(\theta)=\ln \E(e^{\langle\theta|\nX\rangle})$.  If   \eqref{Meixner-Anshelevich} holds, then
\begin{equation}
  \label{k-equation}
  \Psi(k''(\theta))(\mI_n)=\mI_n+2 b k'(\theta)+4a (k'(\theta))^2.
\end{equation}

Conversely, if  the logarithm of the Laplace transform of $\nX$ satisfies \eqref{k-equation} (together with the
initial conditions) then  \eqref{Meixner-Anshelevich} holds, and   $\E (\nX)=\mO$, $\E (\nX^2)=\mI_n$.
\end{corollary}
\begin{proof}
We apply \eqref{pre-k-eqtn} with $A=2a/(1+2a)$, $B=2b/(1+2a)$, $C=2/(1+2a)$.

\end{proof}

\section{Examples of Meixner Ensembles}\label{Sect:Meixner_Ensemble}
In this section we give examples of Meixner ensembles.  In Theorem \ref{T.U} we will show that the examples exhaust the family of all Meixner ensembles on $\HH_{2,\beta}$ with $\Theta_\mu\ne \emptyset$.

\subsection{Bernoulli and binomial ensembles}
The Bernoulli ensembles (which are different from the Bernoulli random matrices with independent two-valued entries) %
will be our basic building blocks for  more complicated
ensembles. Suppose $\mP_m$ is the  orthogonal projection onto the random and uniformly distributed $m$-dimensional
subspace, {\em i.e.}  $\mP_m=\mU \mI_{m,n} \mU^*$ where $\mI_{m,n}$ is the $n\times n$ matrix with $m$ ones on the
diagonal followed by $n-m$ zeroes, and $\mU$ is uniformly distributed on the group $\mathcal{K}_{n,\beta}$.

\begin{definition} %
Denote by $\mu_m$ the distribution of $\mP_m$. A Bernoulli ensemble with parameters $q_1,\dots,q_n\geq 0$ %
such that $q_1+\dots+q_n\leq 1$,
 is the law
$$\mu=q_0\delta_{\mO}+q_1\mu_1+\dots+q_{n-1}\mu_{n-1}+q_n\delta_{\mI_n}\,,$$
where $q_0=1-(q_1+\dots+q_n)$  %
and $\delta_\mI$ denotes a point-mass at $\mI$.
\end{definition}
 In terms of random variables, Bernoulli ensemble can be realized by the random variable
\begin{equation}
  \label{Eq:PtoX}
 \mX=\begin{cases}
\mO & \mbox{with probability $q_0=1-(q_1+\dots+q_n)$}\,,\\
\mP_1 & \mbox{with probability $q_1$}\,,\\
\vdots & \\
\mI_n & \mbox{with probability $q_n$}\,.
\end{cases}
\end{equation}

 \begin{proposition}\label{P-Bern}
A Bernoulli ensemble is a Meixner ensemble with parameters $A=-1$, $B=2$, $C=0$.
\end{proposition}
\begin{proof}
We note that for any pair of projections
\begin{equation}\label{(1)}
(P-Q)^2=2(P+Q)-(P+Q)^2.
\end{equation}
Applying this algebraic identity  to random projections $\mX$ and $\mY$, where $\mX$ is given by \eqref{Eq:PtoX}
and $\mY$ is its independent copy, we see that \eqref{pre-ME} holds.
\end{proof}

 We note that for the Bernoulli ensemble formula \eqref{pre-L-eq}  takes a particularly simple form:  its Laplace
transform $L_B(\theta)=\E (\exp(\langle\theta|\mX\rangle))$ satisfies
\begin{equation}\label{LT-Bern}
\Psi(L''_B(\theta))(\mI_n)=L_B'(\theta).
\end{equation}
 The corresponding version of \eqref{pre-k-eqtn} is
\begin{equation}\label{k-Bern}
\Psi(k''_B(\theta))(\mI_n)=k'_B(\theta)(\mI_n-k'_B(\theta)).
\end{equation}

The fact that all Bernoulli ensembles share the same $A,B,C$ will be  convenient for calculations. But it will be
easier to classify the resulting Meixner laws into the familiar families from the  standardized form that appears
in  formula \eqref{Meixner-Anshelevich}.  To do so, we find the moments of $\mX$. We first note that by rotation
invariance, $\E(\mP_k)=c\mI_n$, so $\E \tr(\mP_k)=n c$. But $\tr(\mP_k)=\tr(\mI_{k,n})=k$, so
\begin{equation}
  \label{Eq:P_moms}
  \E(\mP_k)=\E(\mP_k^2)=\frac{k}{n}\mI_n\,.
\end{equation}
Therefore, with  $\bar{q}=\frac{1}{n}(q_1+2q_2+\dots+nq_n)$,
we have $\E(\mX^2)=\E(\mX)=\bar q\mI_n$.

If $\bar q\ne 0,1$, the standardized matrix is
$\nX=\frac{1}{\sqrt{\bar  q(1- \bar  q)}}(\mX-\bar  q \mI_n)$, and we get
\begin{equation}\label{Bern*}
\E((\nX-\nY)^2|\nS)=4\left(\mI_n+\frac{(1-2 \bar  q) \nS}{2 \sqrt{\bar  q(1-\bar  q)}}-\frac{\nS^2}{4}\right).
\end{equation}
So   \eqref{Meixner-Anshelevich} holds with parameters $a=-1/4$ and  $b=(1/2-\bar  q)/\sqrt{\bar  q(1-\bar  q)}$.

An intrinsic characterization of Bernoulli ensembles with an elementary proof is as follows.
\begin{proposition}\label{P:X^2=X}
If $\mX$ with values in $\HH_{n,\beta}$ has a law invariant under rotations and $\mX=\mX^2$ then its law is  a Bernoulli ensemble.
\end{proposition}
\begin{proof}
Since $\mX$ is symmetric, it is an orthogonal projection. Denote by $d(\mX)$ the dimension of its image. Invariance under rotations means that for any bounded measurable function $f$, we have
$$
\E(f(U\mX U^*)-f(\mX))=0
$$
for all $U\in\calK_{n,\beta}$.
If $d=0,1,\dots,n$ and if $f$ is zero outside of matrices of trace $d$, then
$$
\E(f(U\mX U^*)-f(\mX)|d(\mX)=d)=\E(f(U\mX U^*)-f(\mX))=0.
$$
This implies that the conditional law $\calL(\mX|d(\mX)=d)$ is invariant under rotations, so it is a law of $\mP_d$. Then $\mX$ is a Bernoulli ensemble with parameters $ q_m=\Pr(d(\mX)=m)$.
\end{proof}
This implies the following converse of Proposition \ref{P-Bern}.
\begin{proposition}
Suppose that a law  $\mu$ belongs to a Meixner ensemble with parameters $A=-1$, $B=2$, $C=0$, and that the first four moments are finite. Then $\mu$ is a Bernoulli ensemble.
\end{proposition}
\begin{proof}
Let $\mX,\mY$ be independent random variables with law $\mu$, and let $\mS=\mX+\mY$. Since $(\mX-\mY)^2=2\mX^2+2\mY^2-\mS^2$, and since the law of $(\mX,\mS)$ is the same as $(\mY,\mS)$ from \eqref{pre-ME} we get
$$
\E(\mX^2|\mS)=\frac12\mS.
$$
From $\E(\mX|\mS)=\frac12\mS$ we see that
\begin{equation}\label{XX|S}
\E(\mX^2|\mS)=\E(\mX|\mS).
\end{equation}
From this, we deduce that $\E(\mX)=\E(\mX^2)$, and that
$$
\E(\mX^2\mS)=\E(\mX\mS),
$$
which implies that $\E(\mX^3)=\E(\mX)$.

Using  \eqref{XX|S} again, we have
\begin{equation}\label{trXX|S}
\tr \E(\mX^2\mS^2)=\tr\E(\mX\mS^2).
\end{equation}
Since $\E
\circ \tr = \tr \circ \E$,  from \eqref{trXX|S} using independence and cyclic invariance of the trace we get
$$
\tr \E(\mX^4)+2\tr (\E\mX^3\E\mX)+\tr((\E\mX^2)^2)=\tr \E(\mX^3)+2\tr (\E\mX^2)^2+\tr(\E\mX^2\E\mX).
$$
Since we already proved that   $\E(\mX^3)=\E(\mX^2)=\E(\mX)$, we see that
$\tr \E(\mX^4)=\tr \E(\mX^3)$.

Let $\La_1,\La_2,\dots\La_n$ be the (random) eigenvalues of $\mX$. From the equality of the first four moments of
$\mX$, we see that %
$$
\sum_{j=1}^n\E(\La_j(1-\La_j))^2=\tr \E(\mX^2(\mI_n-\mX)^2)=\tr \E(\mX^2)+\tr \E(\mX^4)-2\tr \E(\mX^3)=0.
$$
Therefore, all eigenvalues of $\mX$ are either $0$ or $1$, and $\mX^2=\mX$. So by  Proposition \ref{P:X^2=X}, the
law $\mu$ of $\mX$ is a Bernoulli ensemble.
\end{proof}

The J\o rgensen set of a Bernoulli ensemble is of interest due to connections with free probability which allows
continuous values for the analog of parameter $N$. For $n=2$, we will show that the J\o rgensen set is $\NN$. The
following sufficient condition shows that this is a "generic case".
\begin{proposition}\label{P-Jorg-Bern} If $\alpha \in \La(\mu)$ and $\mu$ is a Bernoulli ensemble with parameters $q_1,\dots,q_n$ on $\HH_{n,\beta}$, then $(q_0+q_1z+\dots+q_nz^n)^\alpha$ is a polynomial in variable $z$. In particular, if $q_n>0$ %
then $n \alpha$ is an integer.
\end{proposition}
\begin{proof} For $\theta_s=\mbox{diag}(s,s,\dots,s)$, the Laplace transform is
$L(\theta_s)=\sum_{r=0}^n q_r e^{rs}=\prod_{j} (e^s-z_j)^{m_j}$, where $z_j\in\CC$ are the distinct roots of the
polynomial $q_0+q_1z+\dots+q_nz^n$ taken with their multiplicities $m_j$.  Since the Laplace transform must be an
analytic function on its domain, we see that $\alpha$ must be rational, and that if $\alpha=p/q$ with relatively
prime $p,q\in\NN$, then $q$ must be a common divisor of all multiplicities $m_1,m_2,\dots$ of the roots. So $q$
divides also their sum $m_1+m_2+\dots=n$.  \end{proof}

\subsubsection{Binomial  ensemble}
The real, complex or quaternionic binomial ensembles  are measures on   $\HH_{n,\beta}$ which are  parametrized by
integer $N$ and a discrete probability law on $\{0,1,\dots,n\}$. We will follow the tradition of not listing the
probability of $0$ among the parameters.

Fix integer $N$    and non-negative numbers $ q_1,\dots, q_n$ with $ q_1+\dots+ q_n\leq 1$. Let
$\mX_1,\dots,\mX_N$ be independent random matrices with the same
Bernoulli  distribution \eqref{Eq:PtoX}. %
\begin{definition}\label{Def-Bin}
The binomial ensemble  Bin$(N, q_1,\dots, q_n)$ with parameters $N=1,2,\dots $ and $ q_1,\dots, q_n$, is the law of $\mX=\sum_{j=1}^N \mX_j$.
\end{definition}

\begin{proposition}\label{P.B0}
A binomial ensemble with parameter $N$ is a Meixner ensemble
with parameters
$A=-1/(2N-1)$, $ B=2N/(2N-1)$, $C=0$.
\end{proposition}
\begin{proof}
This is a special case of Proposition \ref{P2.2}, applied to the law $\mu$ of $\mX_1$ with parameters described in
Proposition \ref{P-Bern}, and to $\alpha=N$.

\end{proof}

For standardization, we will need to know that
\begin{equation}\label{Bin-moments}
\E(\mX)=N\bar q  \mbox{ and } \Var(\mX)=N\bar q(1-\bar q) ,
\end{equation}
where $\bar q$ is defined just below \eqref{Eq:P_moms}.
From \eqref{ABC2ab} we then get that \eqref{Meixner-Anshelevich} holds with $a=-1/(4N)$ and $b=(1/2-\bar  q)/\sqrt{N\bar q (1-\bar  q)}$.

\begin{remark}Applying Remark \ref{Rem:Exp-Fam} to the Bernoulli and binomial ensemble with parameters $q_1,\dots,q_n$ gives the new binomial ensemble with parameters
$q_1(t),q_2(t),\dots,q_n(t)$, where
$
q_j(t)=q_j r^j/P(r)
$,  $r=e^t$ and $P(r) =q_0+q_1r+\dots+q_nr^n$. To see this, we observe that  if $\mu(dx)=q_0\delta_0+q_1\mu_1+\dots+q_n\delta_{\mI_n}$ then
$e^{t\tr x}\mu(dx)=q_0+q_1r \mu_1+\dots+q_nr^n\delta_{\mI_n}$ since $\tr x=j$ on the support of $\mu_j$.
\end{remark}
\begin{remark}
Let us mention here a geometric interpretation of the binomial distribution for $n=2$.
For simplicity we explain this interpretation for $\beta=1$; its extension to $\beta=2,4$ is fairly straightforward, compare \eqref{theta2s}.
All $2\times 2$ real symmetric matrices are parametrized by $(a,b+ic)\in\RR\times\CC$ as
$$\mM=\left[\begin{matrix}a+b&c\\c &a-b
\end{matrix}
\right]\mapsto v_\mM=(a,b+ic).$$
In particular, matrix $\mM$ corresponding to $v_\mM=(a,z)$ is semipositive definite if and only if $a\ge|z|$.
We have $\tr \mM=2a$ and $\det \mM=a^2-|z|^2$.

If $\mP$ is a projection matrix of rank 1 then $v_\mP=\frac12(1,e^{i T})$ and $\mP$ has distribution $\mu_1$ from Definition \ref{Def-Bin} if and only if $ e^{iT}$ is uniformly distributed on the unit circle.
If $\mX_1,\dots,\mX_N$ are i.i.d. with distribution $\mu_1$ then  $\mS_N=\mX_1+\dots\mX_N$ is Bin$(N,1,0)=\mu_1^{*N}$. So we have
$$v_{\mS_N}=\frac12(N,e^{iT_1}+e^{iT_2}+\dots+e^{iT_N}),$$
where $v_{\mX_j}=\frac12(1,e^{iT_j})$.
Therefore the eigenvalues of $\mS_N$ are
$$
\la_{\pm}=\frac12(N\pm |e^{iT_1}+e^{iT_2}+\dots+e^{iT_N}|).
$$
The distribution of $  |e^{iT_1}+e^{iT_2}+\dots+e^{iT_N}|$ was studied by Kluyver and Rayleigh, see \cite[pages 419--421]{Watson:1944}  and  it is known that
$$
P(|e^{iT_1}+e^{iT_2}+\dots+e^{iT_N}|\leq r)=r\int_0^\infty J_1(rt)(J_0(t))^Ndt
$$
for $0\leq r\leq N$ and $N\geq 2$. (Here $J_0$, $J_1$ are Bessel functions.) In particular, $P(\la_+-\la_-<1)=\frac{1}{N+1}$,  see \cite[page 104]{Spitzer:1964} .

The general binomial ensemble with parameters $N$ and $q_1,q_2$ is
$$
\sum_{\nu_0,\nu_1,\nu_2\geq 0, \nu_0+\nu_1+\nu_2=N}
\frac{N!}{\nu_0!\nu_1!\nu_2!}q_0^{\nu_0}q_1^{\nu_1}q_2^{\nu_2}\mu_1^{*\nu_1}*\delta_{\nu_2\mI_2},$$
and the term $\mu_1^{*\nu_1}*\delta_{\nu_2\mI_2}$ is given by the distribution of $v=(\frac{\nu_1}{2}+\nu_2,e^{iT_1}+e^{iT_2}+\dots+e^{iT_N})$.

\end{remark}
\subsection{Poisson  ensemble}
Our Poisson ensembles have parameters $\la_1,\la_2,\dots,\la_n\geq0$ with $\la=\la_1+\dots+\la_n>0$.
\begin{definition}\label{Def-Poiss}
Let $N$ be a Poisson real random variable, $\Pr(N=j)=e^{-\la}\la^j/j!$, $j=0,1,\dots$, and let $
q_1=\la_1/\la,\dots, q_n=\la_n/\la\geq 0$.  Let $\mX_1,\mX_2,\dots$  be independent
Bernoulli matrices with the same parameters $ q_1,\dots, q_n$, and $\mX_0=\mO$.  The Poisson ensemble is the law
of $\mX=\sum_{k=0}^N \mX_k$.

We will use notation Poiss$(\la_1,\dots,\la_n)$.
\end{definition}

 \begin{proposition}\label{P.P0}
The Poisson ensemble is a Meixner ensemble, as it satisfies
\eqref{pre-ME} with parameters $A=0$, $B=1$, $C=0$.
\end{proposition}
\begin{proof}
Let $L_B(\theta)$ be the Laplace transform of the corresponding Bernoulli ensemble. Then the log of the Laplace
transform for the Poisson ensemble is $k(\theta)=\la(L_B(\theta)-1)$. From \eqref{LT-Bern} we
get
$$
\Psi(k''(\theta))(\mI_n)=k'(\theta),
$$
which is \eqref{pre-k-eqtn} with $A=0$, $B=1$, $C=0$. Therefore \eqref{pre-ME} holds by Proposition \ref{P1.2}. %
\end{proof}
For standardization, we will need to know that, with $\bar \la=(\la_1+2\la_2+\dots+n\la_n)/n$,

\begin{equation}\label{Poiss-moments}
\E(\mX)=\bar \la \mI_n  \mbox{ and } \Var(\mX)=\bar \la\mI_n\,.
\end{equation}
Indeed, $\E(\mX)=\E(\E(\mX|N))$ and $$\Var(\mX)=\E(\Var(\mX|N))+\Var(\E(\mX|N))=
\E(N \bar  q(1-\bar  q))\mI_n+\Var(N\bar  q)\mI_n.$$
One can then check that \eqref{Meixner-Anshelevich} holds with $a=0$ and $b=\frac{1}{2\sqrt{\bar\la}} $.
It is also easy to see that if the law of $\mX$ is Poisson with parameters $\la_1,\dots,\la_n>0$, then $\mX=\sum_{m=1}^n\mX_m$ is the sum of  independent random variables $\mX_m$ from Poisson ensembles with parameters $(0,\dots,0,\la_m,0\dots,0)$. Furthermore, real random variable $\tr(\mX)$ has the compound Poisson law with the law of summands given as $\sum_{k=0}^n \frac{\la_k}{\la}\delta_k$.

 We remark that the Poisson model with exactly one parameter $\la_1\ne 0$  appears explicitly in \cite[page 638]{Cabanal-Duvillard:2005}, as part of the construction of matrix models for all free-infinitely divisible laws; see also \cite{Benaych-Georges04}. However, regression properties of the model were not analyzed.

\subsection{Negative binomial ensemble}
We now use  Bernoulli ensembles to construct the negative binomial ensemble.
Let $N$ be a negative binomial random variable with parameters $r>0$,  $0<p<1$, {\em i.e.}
\begin{equation}\label{NB-1}
P(N=j)= \frac{\Gamma(r+j)}{\Gamma(r) j!} p^r q^j, \; j=0,1,2,\dots,\; q=1-p.
\end{equation}
\begin{definition}\label{Def-NB}
The negative binomial ensemble NB$(r,q_1,\dots,q_n)$ with parameters $r>0$ and $q_1,\dots,q_n\geq 0$ such that $q_1+\dots+q_n<1$,  is the law of the random sum
$$\mX=\sum_{k=0}^N \mX_k\,,$$
where $N$ has distribution \eqref{NB-1} with $p=1-(q_1+\dots+q_n)$, $\mX_1,\mX_2,\dots,$ are independent Bernoulli ensembles with parameters $q_1/(1-p),\dots,q_n/(1-p)$, and $\mX_0=\mO$. \end{definition}

\begin{proposition}\label{P.NB0}
The negative binomial ensemble is a Meixner ensemble, as it satisfies \eqref{pre-ME} with parameters $A=\frac{1}{2 r+1}$, $B= \frac{2r}{2 r+1}$, $C=0$.
\end{proposition}
\begin{proof}
Let $L_B(\theta)$ be the Laplace transform of the corresponding Bernoulli ensemble.
The generating function  $\E (z^N)=p^r/(1-qz)^r$ gives
\begin{equation}\label{Lap-BN}
L(\theta)=\E  (e^{\langle \theta |\mX\rangle } )=p^r \left(1-q L_B(\theta)\right)^{-r},
\end{equation}
so
$$
1-q L_B(\theta)=p (L(\theta))^{-1/r}.
$$
Differentiating, we get
$$
q L_B'(\theta)=\frac{p}{r} L'(\theta) (L(\theta))^{-(1+r)/r}
$$
and
$$q L^{''}_B(\theta)=\frac{p}{r}L''(\theta) (L(\theta))^{-(1+r)/r}-\frac{p(1+r)}{r^2} L'(\theta)\otimes L'(\theta) (L(\theta))^{-(1+2r)/r}.$$
From \eqref{LT-Bern}, we get
$$
\frac{1}{r}\Psi(L''(\theta)(\mI_n))L(\theta)=\frac{1+r}{r^2} (L'(\theta))^2+\frac{1}{r}L'(\theta)L(\theta).
$$
We re-write this as
$$\frac{2 r}{2 r+1}\Psi(L''(\theta)(\mI_n))L(\theta)=\frac{2 (r+1)}{2 r+1}(L'(\theta))^2+\frac{2 r}{2 r+1}L'(\theta)L(\theta),$$
which is \eqref{pre-L-eq} with  $A=\frac{1}{2 r+1}$, $B= \frac{2r}{2 r+1}$,  $C=0$.    Therefore \eqref{pre-ME}
holds by Proposition \ref{P1.2}.
\end{proof}

For standardization, we will need to know that with $\bar q=(q_1+2q_2+\dots+nq_n)/n$,

\begin{equation}\label{NB-moments}
\E(\mX)= r \frac{\bar q }{p} \mI_n \mbox{ and } \Var(\mX)= \frac{r \bar q }{p^2}\left(p+\bar q\right)\mI_n\,.
\end{equation}

Indeed, $\E(\mX)=\E(\E(\mX| N))= \bar{q}/(1-p)\E(N)\mI_n=\bar q r/p\mI_n$ and
\begin{multline*}
\Var(\mX)=\E(\Var(\mX|N))+\Var(\E(\mX|N))=
\E(N \bar  q(1-\bar  q))\mI_n+\Var(N\bar  q)\mI_n
\\= \bar  q (1-\frac{\bar  q}{1-p}) \frac{r}{p}\mI_n+\bar  q^2 \frac{r }{(1-p)p^2}\mI_n\,.
\end{multline*}
Then from \eqref{ABC2ab} we see that \eqref{Meixner-Anshelevich} holds with $a=\frac{1}{4r}$ and $b=\frac{p+2\bar
q }{2 \sqrt{r\bar q (p+\bar q)}}$.
 \subsection{Gaussian ensemble}
Since  $\mX-\mY$ and $\mS$ are independent for any Gaussian  independent identically distributed pair $\mX,\mY$,
formula \eqref{pre-ME} holds with $A=B=0$ and $C=2$.  The requirement of rotational invariance reduces the choices
of the Gaussian law to the three-parameter family:
\begin{equation}\label{k-nn}
k(\theta)=c_1\tr \theta+c_2(\tr \theta)^2/2+c_3\tr(\theta^2)/2
\end{equation}
with $c_3\geq 0$ and $n c_2+c_3\geq 0$.
To see this, without loss of generality we take $c_1=0$. Let $\theta_1,\dots,\theta_n$ be the eigenvalues of
$\theta$, and consider the matrix of the quadratic form
$c_2(\theta_1+\dots+\theta_n)^2+c_3(\theta_1^2+\dots+\theta_n^2)$ whose characteristic polynomial is
$(z-(nc_2+c_3))(z-c_3)^{n-1}$. This shows that $c_3\geq 0$ and $n c_2+c_3\geq 0$.

The explicit construction  is to start with auxiliary GUE/GOE/GSE matrix $\mZ$ and independent real standard
normal $\zeta$. \begin{definition}\label{Def-Gauss} For $c_1\in\RR$,
$c_3\geq 0$ and $c_2\in\RR$ such that $nc_2+c_3\geq 0$, the Gaussian Meixner ensemble is the law of
$$
\mX=\sqrt{c_3}\left(\mZ-\frac1n\tr(\mZ)\mI_n\right)+\sqrt{c_2+\frac{c_3}{n}} \, \zeta \mI_n+c_1\mI_n\,.
$$
\end{definition}
\begin{proposition}
The logarithm of the Laplace transform of the Gaussian Meixner ensemble is  given by  \eqref{k-nn}.
\end{proposition}
\begin{proof}
Since
$\langle\theta|\mZ-a\tr \mZ\rangle = \langle \theta-a\tr\theta|\mZ\rangle$ and
$\tr((\theta-\frac1n\tr\theta)^2)=\tr\theta^2-\frac1n(\tr \theta)^2$, we see that the answer follows from the well
known GUE/GOE/GSE formula $\E\exp(\tr(\theta \mZ))=\exp(\tr\theta^2/2)$.
\end{proof}

\subsection{Gamma ensemble}
The remaining types of Meixner ensembles %
will be constructed   only for matrices of size $n=2$.
One difficulty we encounter is  lack of continuity, which we now explain.

 It is natural to expect that gamma ensemble arises as  $\lim_{p\to 0} p \mX_p$ of   a sequence of negative binomial ensembles with varying parameter $p$ while $r$ and the ratios $q_j/(1-p)$ are kept fixed.
However from \eqref{Lap-BN} we see that
$$\lim_{p\to 0} \E e^{\langle\theta|p\mX_p\rangle}=\lim_{p\to 0} p^r \left(1-q L_B(\theta)\right)^{-r}=\left(1-\tr(\theta L'_B(0))^{-r}=(1-\bar  q \tr\theta)\right)^{-r}.$$
So this is a trivial ensemble of the form $\xi \mI_n$ with real gamma-distributed $\xi$.

We now show that there are   non-trivial gamma ensembles   for $n=2$. This construction is based on a more detailed analysis of the system of PDEs that arises from \eqref{k-equation}.

\begin{definition}\label{Def-Gamma} The Gamma ensemble on $\HH_{2,\beta}$ with parameters $p>\beta/2$, $c>1$ is defined by its Laplace transform
\begin{equation}\label{Gamma-function}
\E (e^{\langle \theta |\mX\rangle })=\left( 1-4c\sqrt{1+\beta} \tr\theta+\beta \tr^2\theta +4\det\theta\right)^{-p},
\end{equation}
defined on
$$\Theta_\mX=\{\theta\in\HH_{2,\beta}:  1-4c\sqrt{1+\beta} \tr\theta+\beta \tr^2\theta +4\det\theta>0\}.$$
\end{definition}
Of course, this definition requires a proof that the required law on $\HH_{2,\beta}$ exists. The Gamma ensemble on
$\HH_{2,\beta}$  is constructed by   choosing the appropriate law on $\RR^{\beta+2}$, and arranging the
corresponding real random variables into the random  matrix. The construction is based on  the laws analyzed by
Letac and Weso\l owski  \cite[Theorem 3.1]{Letac-Wesolowski:2008}.
According to this result,  for $p>\beta/2$ and $c>1$ there is a probability measure $\nu_{p,c}(dx)$ on the open
Lorentz cone $\Omega_\beta=\{x\in\RR^{\beta+2}: x_0>\sqrt{\sum_{r=1}^{\beta+1} x_r^2} \}$ with the Laplace
transform %
 \begin{equation}\label{Gindikin}
\int_{\Omega_\beta} e^{\sum s_rx_r}\nu_{p,c}(dx)=\left(1-2cs_0+s_0^2-\sum_{r=1}^{\beta+1}s_r^2\right)^{-p}.\end{equation}

\begin{proposition} %
Consider $(\xi_0,\dots,\xi_{\beta+1})$ with joint distribution $\nu_{p,c}$. The gamma ensemble with parameters $p,c$ is:
\begin{enumerate}
\item for $\beta=1$,
\begin{equation}\label{X-Gamma1}
\mX=\macierz{ \sqrt{1+\beta}\xi_0 +\xi_1 & \xi_2\\
 \xi_2& \sqrt{1+\beta}\xi_0 -\xi_1
}.
\end{equation}
\item for $\beta=2$,
\begin{equation}\label{X-Gamma2}
\mX=\macierz{ \sqrt{1+\beta}\xi_0 +\xi_1 & \xi_2+i\xi_3\\
 \xi_2-i\xi_3& \sqrt{1+\beta}\xi_0 -\xi_1
},
\end{equation}
where $i=\sqrt{-1}$.
\item for $\beta=4$,
\begin{equation}\label{X-Gamma}
\mX=\macierz{ \sqrt{1+\beta}\xi_0 +\xi_1 & \xi_2+i\xi_3+j\xi_4+k\xi_5\\
 \xi_2-i\xi_3-j\xi_4-k\xi_5& \sqrt{1+\beta}\xi_0 -\xi_1
},
\end{equation}
where  $i,j,k$ are the standard quaternion basis.
\end{enumerate}

\end{proposition}
\begin{proof}
We give the proof for $\beta=4$. Writing %
\begin{equation}\label{theta2s}
\theta=\macierz{s_0 +s_1& s_2+is_3+js_4+ks_5\\
s_2-is_3-js_4-ks_5 & s_0-s_1}
\end{equation}
we have %
$\langle\theta|\mX\rangle=2\sqrt{5}s_0\xi_0+2\sum_{r=1}^{5}s_r\xi_r$,
so \eqref{Gamma-function} follows  from \eqref{Gindikin}.
\end{proof}

\begin{proposition}\label{Gamma-Laplace}
A gamma ensemble on $\HH_{2,\beta}$ with parameters $p>\beta/2$, $c>1$  is a Meixner ensemble with
parameters
$$A=\frac{1}{1+2p},\; B=0,\; C=0.$$
The mean and the variance are $\E (\mX)=2pc\sqrt{1+\beta}\,\mI_2$, $\Var (\mX)=4 p c^2(1+\beta)\mI_2$.
For the standardized version, \eqref{Meixner-Anshelevich} holds with $a=1/(4p)$ and $b=1/\sqrt{p}$.
\end{proposition}
The proof relies on the system of PDEs derived in Theorem \ref{T1}; it appears after Proposition \ref{P6.1}.

\begin{remark}\label{R:gamma-n}
Somewhat more generally, one can define gamma ensembles on $\HH_{n,\beta}$ as the ensembles with Laplace transform
\begin{equation}
L(\theta)=(1-c \tr \theta+(\beta (n-1)+2)(\tr \theta)^2-2\tr\theta^2)^{-p}.
\end{equation}
Such an ensemble is a Meixner ensemble, as it follows from Theorem  \ref{T1} that \eqref{Meixner-Anshelevich} holds with $a=1/(4p)$, $b=1/\sqrt{p}$.
Since this is only  a one-parameter subset of the possible solutions,  we do not pursue this construction further.
\end{remark}

\subsection{Exceptional hyperbolic ensemble}\label{Sect-EHM}
It turns out that there is just one exceptional class of non-trivial hyperbolic ensembles
that, unlike in all previous cases, cannot be "continuously deformed" into the trivial ensemble.

Recall that the Bessel functions and the modified Bessel functions are
\cite[\S3.1 (8), \S3.7 (2)]{Watson:1944} %
\begin{equation}\label{J_and_I}
\begin{gathered}
J_\nu(z)=\sum_{n=0}^\infty \frac{(-1)^n(z/2)^{\nu+2n}}{n! \, \Gamma(n+\nu +1)},
\quad
I_\nu(z) = \sum_{n=0}^\infty \frac{(z/2)^{\nu+2n}}{n! \, \Gamma(n+\nu +1)}\,.
\end{gathered}
\end{equation}
Following  \cite{Koornwinder:2010}, we work with
 differently normalized Bessel functions
\begin{eqnarray}
\label{G} \calJ_\nu(x)&:=&\Gamma(\nu+1)\,(2/x)^\nu\,J_\nu(x),\\
 \calI_\nu(x)&:=&\Gamma(\nu+1)\,(2/x)^\nu\,I_\nu(x).
\label{F}
\end{eqnarray}
which are entire  functions of $z\in\CC$ for $\nu>-1$.

\begin{lemma}\label{P-J-transff}
For every $\alpha>0$,  $\nu\geq 0$  and  $m\in\NN$,  there exist a unique probability
measure $\mu_\alpha$ on $ \RR^m$ with the Laplace transform
\begin{equation}\label{Hyp-2}
\int_{\RR^m}e^{\langle  {\mathbf t}, {\mathbf y}\rangle}\mu_\alpha (d\mathbf{y})
=\frac{1}{\left(\calJ_{\nu}(\|\mathbf{t}\|)\right)^{\alpha}}
\end{equation}
for all  $\mathbf{t}\in\RR^m$ small enough.
\end{lemma}
\begin{proof}   According to \cite[\S15.27]{Watson:1944}, Bessel function $J_{\nu}$ has simple real zeros that come in opposite pairs.  Arranging the positive zeros in  increasing order, 
 $0<j_1<j_2<\dots$, from   \cite[\S15.41 (3)]{Watson:1944} we get 
 $$
 \calJ_{\nu}(z)=\prod_{k=1}^\infty (1-z^2/j_k^2),
 $$
 so for $0<r<j_1$, %
 $$\frac{1}{\left(\calJ_{\nu}(\sqrt{r})\right)^\alpha}=\prod_{k=1}^\infty \frac{1}{(1-r /j_k^2)^\alpha}$$
 is a Laplace transform (of the  series of independent  Gamma random variables).
 Therefore, by Schoenberg's theorem   \cite[Theorem 2]{Schoenberg:1938a},
 for every $m$ there exists a measure $\mu_\alpha$ on Borel sets of $\RR^m$ such that for $\|\mathbf{t}\|<j_1$,
 $$\frac{1} {\left(\calJ_{\nu}(\|\mathbf{t}\|)\right)^\alpha}=\int_{\RR^m}\exp\langle \mathbf{t},\mathbf{y}\rangle \mu_\alpha(d\mathbf{y}).
$$
\end{proof}
The hyperbolic Meixner ensembles have two "disconnected" components: the two-parameter family of trivial ensembles $\xi\mI_2$, and the following "exceptional" one-parameter non-trivial ensemble.
\begin{definition}\label{P3.11}
For $\alpha>0$, consider $(\xi_1,\dots,\xi_{\beta+1})$ with joint distribution $\mu_\alpha$ from Lemma \ref{P-J-transff}. The exceptional hyperbolic ensemble  on $\HH_{2,\beta}$ with parameter $\alpha$ is:
\begin{enumerate}
\item for $\beta=1$,
\begin{equation*}\label{X-Hyperbolic1}
\mX=\macierz{  \xi_1 & \xi_2\\
 \xi_2 &   -\xi_1
},
\end{equation*}
\item for $\beta=2$,
\begin{equation*}\label{X-Hyperbolic2}
\mX=\macierz{  \xi_1 & \xi_2+i\xi_3 \\
 \xi_2-i\xi_3 &   -\xi_1
},
\end{equation*}
where $i=\sqrt{-1}$,
\item for $\beta=4$,
\begin{equation}\label{X-Hyperbolic}
\mX=\macierz{  \xi_1 & \xi_2+i\xi_3+j\xi_4+k\xi_5\\
 \xi_2-i\xi_3-j\xi_4-k\xi_5&   -\xi_1
},
\end{equation}
where $i,j,k$ is the standard basis of quaternions.
\end{enumerate}
\end{definition}
We remark that the Laplace transform of the exceptional hyperbolic ensemble is
\begin{equation}
  \label{H-function}
\E (e^{\langle \theta |\mX\rangle })=\left(\calJ_{(\beta-1)/2}\left(\sqrt{\tr^2\theta-4\det \theta}\,\right)\right)^{-\alpha},
\end{equation}
defined for all $\theta$ with eigenvalues $|\theta_1-\theta_2|<\ell_\beta$, the first positive zero of $\calJ_{(\beta-1)/2}$.
To see this, we use \eqref{theta2s} (or its appropriate modifications for $\beta<4$) with $\langle\theta|\mX\rangle=2 s_0\xi_0+2\sum_{r=1}^{\beta+1}s_r\xi_r$. From \eqref{Hyp-2} with $s=2s_0$ and $\mathbf{t}=(2s_1,2s_2,\dots,2s_{\beta+1})$ we get
\begin{equation}\label{H*H*}
\E (e^{\sum_{j=0}^{\beta+1}2s_j\xi_j})=\left(
 \calJ_{(\beta-1)/2}(\left\|\mathbf{t}\right\|)\right)^{-\alpha},
\end{equation}
which gives \eqref{H-function}. (We also used  the elementary observation $|\theta_1-\theta_2|=\sqrt{\tr^2\theta-4\det \theta}$.)

\begin{proposition}\label{H-Laplace}
The  exceptional hyperbolic ensemble with parameter  $\alpha>0$  is a Meixner ensemble with parameters
$$
A=\frac{1}{1+2\alpha}, \;B=0, \;C=\frac{4 \alpha ^2}{ 2 \alpha +1}.
$$
The mean and the variance are $\E (\mX)=\mO$, $\Var (\mX)=\alpha\mI_2$.
For the standardized version, \eqref{Meixner-Anshelevich} holds with  $a=1/(4\alpha)$ and $b=0$.
\end{proposition}
Our proof uses the system of PDEs derived in Theorem \ref{T1}, so it appears after Proposition \ref{P6.1}. 

\section{The system of PDEs}\label{Sect:PDE}
In this section, we derive a system of PDEs for the
Laplace transform of
 general Meixner ensembles. As previously, we consider simultaneously the real, complex, and quaternionic  cases.

We introduce some notation. For $\theta\in \HH_{n,\beta}$ and $j=0,1,\ldots,$ we
consider $\sigma_j(\theta)$ defined by
\begin{equation}
  \label{DEF sigma}
   \det(\mI_n+z \theta)=\sum_{j=0}^n   \sigma_j(\theta) z^j .
\end{equation}
Recall that for $\beta=4$, the determinants of quaternionic matrices are defined only for hermitian matrices, see \cite[page 29]{Faraut-Koranyi-94} and \cite{Casalis:1990}.
Note that
$$ \sigma_0=1,\ \sigma_1(\theta)=\tr  \theta,\
\sigma_2(\theta)=\frac{1}{2}(\tr \theta)^2-\frac{1}{2}\tr(
\theta^2),\ \sigma_n(\theta)=\det \theta,$$
and that $\sigma_j=0$ for  $j>n$. We also adopt the
convention that $\sigma_j=0$ for $j<0$. In general,
$\sigma_j(\theta)=e_j(\theta_1,\theta_2,\dots)$ is the $j$-th elementary
symmetric function of the eigenvalues $\theta_1,\theta_2,\dots$ of
$\theta\in\HH_{n,\beta}$; this notation and most of our calculations
do not depend on the dimension $n$.
Denote by $U$ the open subset of   $\RR^n$ such that if $
(\sigma_1,\ldots,\sigma_n)\in U$ then the polynomial in $x$ defined
by  $\sum_{i=0}^n(-1)^{n-i}\sigma_i x^{n-i}$ has only
distinct real roots.

The main result of this section is the system of PDEs that determines the Laplace transform of a Meixner ensemble on a nonempty open set $\Theta_\mX\subset\HH_{n,\beta}$. This system is written  in terms of an auxiliary function $g(\sigma_1,\dots,\sigma_n)$ which is defined on an appropriate non-empty open set $U_\mX\subset U$. The link between $L$ and this auxiliary function $g$ varies according to the fact that $a$ or $b$ are zero or not.
To define the set $U_\mX$, we need $\Theta_\mX$ to be closed  under conjugation ($\theta\in\Theta_\mX$ implies $U\theta U^*\in\Theta_\mX$ for all $U\in\mathcal{K} _{n\beta}$), and to  have $\mO$ in its closure.
Consider the $n$-dimensional real subspace $D_n\subset\HH_{n,\beta}$ consisting of   diagonal matrices, and let $D_n^0\subset D_n$ denote the set of matrices with distinct eigenvalues. Since $\Theta_\mX$ is invariant under rotations, $\Theta_\mX\cap D_n$ is a nonempty open set with $\mO$ in its closure. The set $\Theta_\mX\cap D_n^0$ is obtained from  $\Theta_\mX\cap D_n$ by removing a finite number of hyperplanes, so it is also non-empty, and has $\mO$ in its closure.

Denote by $\sigma:\HH_{n,\beta}\to \RR^n$  the mapping  $\theta\mapsto (\sigma_1(\theta),\dots,\sigma_n(\theta))$.
 If $\theta\in D_n^0$ then by the Vandermonde's determinant, $\mathcal{B}_\theta=\{\mI_n,\theta,\theta^2,\dots,\theta^{n-1}\}$ is a basis of $D_n$.
From formula \eqref{CH1} below, we see that in that basis, the matrix representation of the derivative of  $\sigma$ restricted to $D_n$ is triangular at $\theta\in D_n^0$, with $\sigma_0=1$ on the diagonal,
$$
[\sigma'(\theta)]_{\mathcal{B}_\theta}=\left[\begin{matrix}
\sigma_0 &0 &0 &\dots   &0 \\
-\sigma_1(\theta)&\sigma_0&0&\dots&0\\
\sigma_2(\theta)&-\sigma_1(\theta)&\sigma_0& \dots&0 \\
\vdots & & & \ddots&\\
(-1)^{n-1}\sigma_{n-1}(\theta)&(-1)^{n-2}\sigma_{n-2}(\theta)&\dots&\dots&\sigma_0
\end{matrix}\right]
$$

Therefore, the Jacobian of $\sigma$ is $1$ and $\sigma:D_n^0\to\RR^n$ is an open mapping. Thus,  the set
\begin{equation}\label{U_X}
U_\mX=\sigma(\Theta_\mX\cap D_n^0)
\end{equation}
is open, non-empty, and since $\sigma(\mO)=0$, it has $0\in\RR^n$ in its closure.

To shorten the formulas,  for $i,j=1,2,\ldots,n$ we write
$$g_i=\frac{\partial g}{\partial\sigma_i}
\mbox{ and } g_{ij}=\frac{\partial^2 g}{\partial\sigma_i\partial\sigma_j}.$$
For $g:U_\mX\to\RR$ consider the vector
\begin{equation}
\mathbb{D}(g)=\left[
\begin{matrix}
 g_{11}-\displaystyle\sum_{r,s=2}^{n}
  \sigma_{r+s-2}g_{rs}-(n-1)\frac{\beta}{2}g_{2} \\
     \sigma_{1}g_{22}+2g_{12}-\displaystyle\sum_{r,s=3}^{n}
  \sigma_{r+s-3}g_{rs}-(n-2)\frac{\beta}{2}g_{3}  \\
   \vdots  \\
  \displaystyle \sum_{r,s=1}^{j} \sigma_{r+s-1-j}g_{rs}-\sum_{r,s=j+1}^{n}
  \sigma_{r+s-1-j}g_{rs}-(n-j)\frac{\beta}{2}g_{j+1} \\
  \vdots   \\
  \displaystyle  \sum_{r,s=1}^{n-1} \sigma_{r+s-n}g_{rs}-
  \sigma_{n}g_{nn}-\frac{\beta}{2}g_{n}  \\
  \displaystyle \sum_{r,s=1}^{n} \sigma_{r+s-n-1}g_{rs}
\end{matrix}
\right]
\end{equation}
Recall  the convention that $\sigma_j=0$ for $j<0$ or
$j>n$.

\begin{theorem}
  \label{T1}   Let $\mX$ be a random variable with values in $\HH_{n,\beta}$ with the law  invariant under rotations. %
  Let $L(\theta)=e^{k(\theta)}=\E(e^{\langle\theta|\mX\rangle})$ be its Laplace transform, and assume that $\Theta_\mX\ne 0$. Suppose that $\E
(\mX)=\mO$, $\E (\mX^2)=\mI_n$ and that $\mX$ satisfies \eqref{Meixner-Anshelevich} with some $a$ and $b$.  Let $U_\mX$ be the associated open subset of $\RR^n$, see \eqref{U_X}.
  \begin{enumerate}
  \item[(i)] If $a\ne 0$, define $g:U_\mX\to\RR$ by
  $ g(\sigma_1(\theta),\ldots,\sigma_n(\theta))=  \exp(-4 a k(\theta)-b \tr\theta).$
    Then $g$  satisfies  the system of PDEs
$$
\mathbb{D}(g)=\left[\begin{matrix}
(b^2-4a)g\\ 0 \\ \vdots \\ 0
\end{matrix}
\right]
$$
Furthermore,
\begin{equation}
  \label{S I ini g}
  g(\sigma_1,\dots,\sigma_n)\to 1, \; \frac{\partial g(\sigma_1,\dots,\sigma_n)}{\partial \sigma_1}\to -b
\end{equation}
as $(\sigma_1,\dots,\sigma_n)\to 0$ over $U_\mX$. %

   \item[(ii)] If $a=0,b\ne 0$,   define $g:U_\mX\to\RR$ by
 $   g(\sigma_1(\theta),\ldots,\sigma_n(\theta))=k(\theta) +\frac{1}{2b} \tr\theta .$
 Then $g$  satisfies the system of PDEs
$$
\mathbb{D}(g)=2b\left[\begin{matrix}
g_1\\ g_2 \\ \vdots \\ g_n
\end{matrix}
\right]
$$

Furthermore,
\begin{equation}
  \label{S II ini g}
  g(\sigma_1,\dots,\sigma_n)\to0,\; \frac{\partial g(\sigma_1,\dots,\sigma_n)}{\partial \sigma_1}\to \frac{1}{2b}
\end{equation}
as $(\sigma_1,\dots,\sigma_n)\to 0$ over $U_\mX$.

  \item[(iii)] If $a=0,b=0$,  define $g:U_\mX\to\RR$ by
  $g(\sigma_1(\theta),\ldots,\sigma_n(\theta))=k(\theta)$. Then $g$ satisfies the system of PDEs
 $$
\mathbb{D}(g)=\left[\begin{matrix}
1\\ 0 \\ \vdots \\ 0
\end{matrix}
\right]
$$
Furthermore,
\begin{equation}
  \label{S III ini g}
  g(\sigma_1,\dots,\sigma_n)\to0,\; \frac{\partial g(\sigma_1,\dots,\sigma_n)}{\partial \sigma_1}\to 0
\end{equation}
as $(\sigma_1,\dots,\sigma_n)\to 0$ over $U_\mX$.
  \end{enumerate}
  Conversely, let $\mu$ be an ensemble on $\HH_{n,\beta}$   such that  $\Theta_\mu\ne\emptyset$. If   one can find constants $a,b$ that lead to one of the equations in (i), (ii) or (iii), then the corresponding version of equation \eqref{k-equation} holds for all $\theta\in\Theta_\mu$, and hence $\mu$ is a Meixner ensemble with parameters $a,b$.
\end{theorem}
The left hand side $\mathbb{D}(g)$ of our systems of the PDE's  resembles the system of PDE's for the Bessel function of matrix argument given in  \cite[Eqtn (18)]{James:1955} when $\beta=1$, see also \cite[page 275 Eqtn (31)]{Muirhead:1982}.
 However, as we will see in Proposition \ref{P6.1}, in contrast to \cite{James:1955}  our system has multiple analytic solutions.

\subsection{Proof of Theorem \ref{T1}}

 The Cayley Hamilton theorem (for the case $\beta=4$, see \cite[Proposition II.2.1]{Faraut-Koranyi-94}) implies
\begin{equation}\label{CH}
\theta^n=\sum_{i=1}^n(-1)^{i-1}\sigma_i(\theta)\theta^{n-i}.\end{equation}
We use also without proof the Newton formula \cite[\S X.4]{Arnaudies:1987} %
\begin{equation}\label{CH0}m\sigma_m(\theta)=\sum_{i=0}^{m-1}(-1)^{m-1-i}\sigma_i(\theta)\tr(\theta^{m-i});\end{equation}
we remark that under the convention that $\sigma_m=0$ for $m>n$,
formula \eqref{CH0} holds true for all $m\geq 1$, {\em i.e.} also for $m>n$.
 Since $\tr \mI_n =n$, it can be rewritten  as
\begin{equation}\label{CH00}(m-n)\sigma_m(\theta)=\sum_{i=0}^{m}(-1)^{m-1-i}\sigma_i(\theta)\tr(\theta^{m-i}).\end{equation}

We also need  a differentiation formula.
\begin{proposition} Fix $n\geq 1$. Then for all $m\geq  1$,
\begin{equation}\label{CH1}\sigma'_m(\theta)=\sum_{i=0}^{m-1}(-1)^{m-1-i}\sigma_i(\theta)\theta^{m-1-i}.\end{equation}
\end{proposition}
\begin{proof} We prove \eqref{CH1} by induction with respect to $m$. Since
$\sigma'_1=\mI_n=\sigma_0\theta^0$, the formula holds true for $m=1$.
Suppose \eqref{CH1} holds for some $m\geq 1$. With
$t_j=\tr(\theta^j)$, we differentiate Newton's formula \eqref{CH0}
written for $\sigma_{m+1}$ and use the induction assumption. We get
\begin{multline}\label{Ind-Step}
(m+1)\sigma'_{m+1}=(-1)^m(m+1)\theta^m\\
+\sum_{j=1}^m(-1)^{m-j}
\Big((m+1-j)\sigma_j\theta^{m-j}
+t_{m+1-j}\sum_{s=0}^{j-1}(-1)^{j-1-s}\sigma_s\theta^{j-1-s}\Big)
\\
=(-1)^m(m+1)\theta^m+(m+1)\sum_{j=1}^m(-1)^{m-j}\sigma_j\theta^{m-j}+R\,,
\end{multline}
where
\begin{multline*}
  R=-\sum_{j=1}^m(-1)^{m-j}j\sigma_j\theta^{m-j}
+\sum_{j=1}^m(-1)^{m-j}t_{m+1-j}\sum_{s=0}^{j-1}(-1)^{j-1-s}\sigma_s\theta^{j-1-s}\\
=-\sum_{j=1}^m(-1)^{m-j}j\sigma_j\theta^{m-j}
+\sum_{s=1}^m(-1)^{m-s}\theta^{m-s}\sum_{j=0}^{s-1}(-1)^{s-1-j}\sigma_jt_{s-j}\,.
\end{multline*}
Using Newton's formula again, we see that
$\sum_{j=0}^{s-1}(-1)^{s-1-j}\sigma_jt_{s-j}=s\sigma_s$, so $R=0$
and the formula follows from \eqref{Ind-Step}.
\end{proof}

We now  consider separately the following cases, in which we re-write \eqref{k-equation}  by the indicated substitutions:
\begin{enumerate}
  \item[(i)] If $a\ne 0$, then
  $ f(\theta)= \exp(-4 a k(\theta)-b \tr(\theta)) $ solves
  \begin{equation}
    \label{System I}
    \Psi(f''(\theta))(\mI_n)=(b^2-4a) f(\theta)\mI_n\,.
  \end{equation}
  The  limits as $\theta\to \mO$  in $\Theta_\mX$ are:
 \begin{equation}
    \label{System I ini}
f(\theta)\to 1, \; f'(\theta)\to -b \mI_n\,.
  \end{equation}

   \item[(ii)] If $a=0,b\ne 0$, then
  $ f(\theta)=   k(\theta)+\frac{1}{2b} \tr(\theta) $ solves
  \begin{equation}
    \label{System II}
    \Psi(f''(\theta))(\mI_n)=2b f'(\theta)\mI_n\,.
  \end{equation}
  The  limits as $\theta\to \mO$  in $\Theta_\mX$ are:  \begin{equation}
    \label{System II ini}
    f(\theta)\to 0, \; f'(\theta)\to \frac{1}{2b}\mI_n\,.
  \end{equation}

  \item[(iii)] If $a=0,b=0$, then $f(\theta)=k(\theta)$ solves
  \begin{equation}
    \label{System III}
    \Psi(f''(\theta))(\mI_n)=\mI_n\,.
  \end{equation}
  The  limits as $\theta\to \mO$  in $\Theta_\mX$ are: \begin{equation}
    \label{System III ini}
f(\theta)\to 0, \;f'(\theta)\to\mO.
  \end{equation}
  \end{enumerate}
The remaining part of the proof consists of finding a suitable expression for $f$.

Now the fact that the law of $\mX$ is invariant under the rotations is equivalent to $k$ being  invariant under the substitution  $\theta\mapsto U\theta U^*$
for all $U\in\mathcal{K}_{n,\beta}$, and is equivalent to saying that in each of the three cases mentioned above,  there
exists a real analytic function $g$ on $U_\mX$   such that
$$f(\theta)=g(\sigma_1(\theta),\ldots,\sigma_n(\theta))$$
for all $\theta\in\Theta_\mX$ with distinct eigenvalues.
(This is Weyl's formula, see  \cite[Lemma 2.6.1]{Mehta91}.)
 Using \eqref{CH1} we get
\begin{equation}\label{f-g-ini}
f'(\theta)=\sum_{m=1}^n g_m\sigma'_m(\theta)
=\sum_{m=1}^ng_m
\sum_{i=0}^{m-1}(-1)^{m-1-i}\sigma_i(\theta)\theta^{m-1-i}.
\end{equation}

The second derivative of $f$ is trickier: we write $f''=A+B$ with
\begin{eqnarray*}A&=&\sum_{m=1}^n\sum_{i=1}^ng_{mi}\sigma'_m(\theta)\otimes\sigma'_i(\theta),\\
B&=&\sum_{m=1}^ng_m\sigma''_m(\theta).\end{eqnarray*} The calculation
of $\Psi(A)(\mI_n)$ is quick and gives
$$\Psi(A)(\mI_n)=\sum_{m=1}^n\sum_{i=1}^ng_{mi}\sigma'_m\sigma'_i\,.$$ For calculating
$\Psi(B)(\mI_n)$ we prove the surprising formula:

\begin{proposition}
  \label{P3.1}  $\Psi(\sigma''_m)(\mI_n)=\frac{\beta}{2}(m-1-n) \sigma'_{m-1}.$
\end{proposition}

\begin{proof}
We split $\sigma''_m(\theta)$ into $\sigma''_m(\theta)=C_m+D_m$
where
\begin{eqnarray}C_m&=&
\sum_{i=0}^{m-1}(-1)^{m-1-i}\sum_{j=0}^{i-1}(-1)^{i-1-j}\sigma_j\theta^{i-1-j}\otimes\theta^{m-1-i}\nonumber
\\&=&\sum_{j=0}^{m-2}(-1)^{m-j}\sigma_j\sum_{s=0}^{m-2-j}\theta^{s}\otimes\theta^{m-2-j-s},\label {CH2}
\\
D_m&=&\sum_{j=0}^{m-1}(-1)^{m-1-j}\sigma_j[\theta^{m-1-j}]'.\label
{CH3}\end{eqnarray}
Since $z\otimes y+y\otimes x=(x+y)\otimes(x+y)-x\otimes x-y\otimes y$, pairing up the appropriate powers from
 Proposition \ref{P1.1} we get
 \begin{equation}\label {CH4}\Psi\left(\sum_{s=0}^{m-2-j}\theta^{s}\otimes\theta^{m-2-j-s}\right)(\mI_n)=
 (m-j-1)\theta^{m-j-2},\end{equation}
which  leads to the calculation of $\Psi(C_m)(\mI_n):$

\begin{equation}\label {CH5}
\Psi(C_m)(\mI_n)=\sum_{j=0}^{m-2}(-1)^{m-j}\sigma_j(m-j-1)\theta^{m-j-2}.
\end{equation}
Having in mind the calculation of  $\Psi(D_m)(\mI_n)$ we write
$$[\theta^{s}]'(h)=\sum_{i=0}^{s-1}\theta^i h\theta^{s-1-i}=%
\frac{1}{2}\sum_{i=0}^{s-1}
\left[\PP{\theta^i+\theta^{s-1-i}}-\PP{\theta^{i}}-\PP{\theta^{s-1-i}}\right](h).$$
From Proposition \ref{P1.1}, we have
$\Psi(\PP{x})=\left(1-\frac{\beta}{2}\right)\PP{x}+\frac{\beta}{2}x\otimes
x$ which leads to
$\Psi(\PP{x})(\mI_n)=(1-\frac{\beta}{2})x^2+\frac{\beta}{2}x
\tr(x)$ and to
$$\Psi(\PP{x+y}-\PP{x}-\PP{y})(\mI_n)=\left(1-\frac{\beta}{2}\right)(xy+yx)+\beta\frac{x \tr( y)+y \tr (x)}{2}.$$
Thus we can compute
\begin{eqnarray*}\Psi([\theta^{s}]')(\mI_n)&=&\frac{1}{2}\left(1-\frac{\beta}{2}\right)
\sum_{i=0}^{s-1}[\theta^{i}\theta^{s-i-1}+\theta^{s-i-1}\theta^{i}]+\frac{\beta}{4}
\sum_{i=0}^{s-1}[\theta^i \tr(\theta^{s-1-i}) + \theta^{s-1-i}\tr
(\theta^i)]\\&=&\left(1-\frac{\beta}{2}\right)s\theta^{s-1}+\frac{\beta}{2}
\sum_{i=0}^{s-1}\theta^i \tr(\theta^{s-1-i}).\end{eqnarray*}
We
apply this to $s=m-1-j$ and compute $\Psi(D_m)(\mI_n):$
\begin{multline}\label {CH6} \Psi(D_m)(\mI_n)=
\left(1-\frac{\beta}{2}\right)\sum_{j=0}^{m-2}(-1)^{m-1-j}\sigma_j(m-1-j)
\theta^{m-2-j}\\+\frac{\beta}{2}\sum_{j=0}^{m-2}(-1)^{m-1-j}\sigma_j
 \sum_{i=0}^{m-2-j}\theta^i \tr(\theta^{m-2-j-i}).
\end{multline} Note that we replaced $m-1$ by $m-2$ in the first
sum of the right hand side.  We obtain
\begin{eqnarray}\nonumber \Psi(\sigma''_m)(\mI_n)=&& \frac{\beta}{2}\sum_{j=0}^{m-2}(-1)^{m-j}\sigma_j(m-1-j)\theta^{m-2-j}\\&&\nonumber
+
\frac{\beta}{2}\sum_{i=0}^{m-2}(-1)^{i}\theta^i\sum_{j=0}^{m-2-i}(-1)^{m-1-j-i}\sigma_j\tr(\theta^{m-2-i-j})
\\ \label {CH8}=&&\frac{\beta}{2}\sum_{j=0}^{m-2}(-1)^{m-j}\sigma_j(m-1-j)\theta^{m-2-j}
\nonumber
\\&&
+\frac{\beta}{2}\sum_{i=0}^{m-2}(-1)^{i}\theta^i(m-2-i-n)\sigma_{m-2-i}
 \\\label {CH9}
=&&
\frac{\beta}{2}\sum_{j=0}^{m-2}(-1)^{m-j}\sigma_j(m-1-j)\theta^{m-2-j} \\ \nonumber &&
+
\frac{\beta}{2}\sum_{j=0}^{m-2}(-1)^{m-j}\theta^{m-2-j}(j-n)\sigma_{j}
\\ \nonumber
=&&\beta\frac{m-1-n}{2}\sum_{j=0}^{m-2}(-1)^{m-j}\sigma_j\theta^{m-2-j}\\ \label {CH10}
=&&\beta\frac{m-1-n}{2}\sigma'_{m-1}\;.
\end{eqnarray}
In this chain of equalities, the first is obtained by  combining
\eqref{CH5} and \eqref{CH6} and reversing the summation in $i$ and
$j$ in the last sum;  \eqref{CH8} is  obtained by applying
\eqref{CH00} (replacing $m$ by $m-2-i$); \eqref{CH9} is  obtained by
changing the variable $j=m-2-i.$ The last equality  \eqref{CH10}
uses the equality \eqref{CH1}. Thus the Proposition \ref{P3.1} is proved.

\end{proof}

Gathering all terms, using the
change  of variable $m=j-1$ in $\Psi(B)$ and the fact that
$\sigma'_0=0$, the expression $\Psi(f'')(\mI_n)=\Psi(A+B)(\mI_n)$ becomes,
$$\Psi(f'')(\mI_n)=\sum_{m=1}^n\sum_{i=1}^ng_{mi}\sigma'_m\sigma'_i
+\frac{\beta}{2}\sum_{j=1}^{n -1}(j-n)g_{j+1}\sigma'_{j}\,.$$
Since $\sigma_1'(\theta)=\mI_n$,  the three equations \eqref{System I},\eqref{System II} and \eqref{System III} respectively become
\begin{equation}\label{I}
  \sum_{m=1}^n\sum_{i=1}^ng_{mi}\sigma'_m\sigma'_i
+\frac{\beta}{2}\sum_{j=1}^{n -1}(j-n)g_{j+1}\sigma'_{j}=(b^2-4a) g \sigma_1'(\theta),
\end{equation}
\begin{equation}\label{II}
  \sum_{m=1}^n\sum_{i=1}^ng_{mi}\sigma'_m\sigma'_i
+\frac{\beta}{2}\sum_{j=1}^{n -1}(j-n)g_{j+1}\sigma'_{j}=2b\sum_{m=1}^n g_m\sigma'_m(\theta),
\end{equation}
\begin{equation}\label{III}
  \sum_{m=1}^n\sum_{i=1}^ng_{mi}\sigma'_m\sigma'_i
+\frac{\beta}{2}\sum_{j=1}^{n -1}(j-n)g_{j+1}\sigma'_{j}=\sigma_1'(\theta).
\end{equation}

Noting that
$\sigma'_1,\sigma'_2,\ldots,\sigma'_n $ is a basis of the algebra
generated by $\theta$ with distinct eigenvalues, the corresponding systems of PDEs arise by comparing the  coefficients
at $\sigma_j'$ in (\ref{I}-\ref{III}), respectively. To end the proof, we only need to find
 the coefficients $P_j(m,i)$
in the basis expansion
\begin{equation}\label{P-expansion}
\sigma'_m\sigma'_i=\sum_{j=1}^nP_j(m,i)\sigma'_j\,.
\end{equation}
This is accomplished by the following formula.
\begin{proposition}
For all $ r,s\geq 1$,
\begin{equation}\label{[*]}
\sigma'_r(\theta)\sigma'_s(\theta)=\sum_{j=0}^{r-1}[\sigma_j,\sigma_{r+s-1-j}],
\end{equation}
where
$$[f,g]:=f(\theta)g'(\theta)-f'(\theta)g(\theta).$$
\end{proposition}
\begin{proof}
From \eqref{CH1} we get
$\sigma'_{j+1}=\sigma_j\theta^0-\theta\sigma_j'$. Since the algebra
generated by $\theta$ is commutative,
$\sigma'_{j+1}\sigma'_n-\sigma'_j\sigma'_{n+1}=(\sigma_j\theta^0-\theta\sigma_j')\sigma'_n-
\sigma'_j(\sigma_n\theta^0-\theta\sigma_n')$ simplifies to
\begin{equation}
  \label{ind-prod}
  \sigma'_{j+1}\sigma'_n-\sigma'_j\sigma'_{n+1}=[\sigma_j,\sigma_n].
\end{equation}
We now prove the proposition by induction. We assume that
\eqref{[*]} holds for some fixed $r$ and all $s\geq 1$. This is
trivially true for $r=1$ since
$\sigma_1'\sigma_s'=\sigma_s'=[\sigma_0,\sigma_s]$. If \eqref{[*]}
holds true for some $r\geq 1$ and all $s\geq 1$, then by
\eqref{ind-prod} and the induction assumption,
$$
  \sigma_{r+1}'\sigma_s'=\sigma_{r}'\sigma'_{s+1}+[\sigma_r,\sigma_s]=
\sum_{j=0}^{r-1}[\sigma_j,\sigma_{r+s-j}]+[\sigma_r,\sigma_s],
$$
which shows that \eqref{[*]} holds for $r+1$ and all $s\geq 1$.
\end{proof}
\begin{proof}[Proof of Theorem \ref{T1}]
Cases (i)-(iii) correspond to equations (\ref{System I}-\ref{System III}).
The left hand sides of the resulting three equations (\ref{I}-\ref{III})  are then re-written using \eqref{P-expansion}.
From \eqref{[*]} we see that coefficients $P_j(m,i)$ are linear with
respect to $\sigma_1,\sigma_2,\ldots,\sigma_n$. Rewriting
\eqref{[*]} as
$$\sum_j
P_j(r,s)\sigma_j'=
\sum_{j=s}^{r+s-1}\sigma_{r+s-1-j}\sigma'_j-\sum_{j=1}^{r-1}\sigma_{r+s-1-j}\sigma'_j,$$
we get the
 explicit formula
\begin{equation}
  \label{Pmij}
P_j(r,s)=\begin{cases} \sigma_{r+s-1-j} & \mbox{ if } \max\{r, s \}\leq j\leq r+s-1\leq n ,\\
 -\sigma_{r+s-1-j} & \mbox{ if }1\leq j < \min\{r, s\} \mbox{ and } j\leq r+s-1\leq n,\\
 0& \mbox{otherwise}.
\end{cases}
\end{equation}
With this notation equations (\ref{I}-\ref{III}) become the equations asserted in Theorem \ref{T1}.
To derive the corresponding boundary limits, we use \eqref{f-g-ini} and the respective limits \eqref{System I ini}, \eqref{System II ini}, or \eqref{System III ini}.

It is also clear that these calculations can be reversed.
From the equations in listed in (i)--(iii), we see that \eqref{k-equation}  holds for all $\theta\in\Theta_\mX$ with distinct eigenvalues. By uniqueness of analytic extension,  \eqref{k-equation} extends to the whole open set $\Theta_\mX$.
\end{proof}
\begin{remark}
Note that to have a solution of the system in each of the three cases is not a guarantee of existence of an ensemble: the function $L(\theta)$ still have to be a Laplace transform of a probability measure. For example,  for $n=2$ there are solutions of the system for all $a<0$, but according to Theorem \ref{T.U}(i) they are the Laplace transforms of a matrix ensemble only if $-1/(4a)\in\NN$.
\end{remark}

\begin{remark}
Particular Meixner ensemble can be found from the system of PDEs by restricting $g$ to be a function of $\sigma_1,\sigma_2$ only. In other worlds, we look for an ensemble such that the Laplace  transform  $L(\theta)$ is a function  of $\tr \theta, \tr\theta^2$. Such strategy leads us to the gamma ensemble on $\HH_{n,\beta}$ mentioned in Remark \ref{R:gamma-n} and to  Definition \ref {Def-Gauss}  of the Gaussian ensemble. This strategy leads only to trivial ensembles for $b^2\ne 4a$.
 \end{remark}

\section{Meixner ensembles on $2\times2$ matrices}\label{Morris-Meixner}

The goal of this section is to use Theorem \ref{T1} to prove that the constructions from Section \ref{Sect:Meixner_Ensemble} exhaust all Meixner ensembles on $\HH_{2,\beta}$. %
\begin{theorem}\label{T.U}
Up to affine transformations, the only Meixner ensembles on $\HH_{2,\beta}$ with the Laplace transform defined on a non-empty open subset of $\HH_{2,\beta}$ are of the six types, depending on the values of $a,b$ from \eqref{Meixner-Anshelevich}:
\begin{enumerate}
\item If $\mu$ is a Meixner ensemble on $\HH_{2,\beta}$ with $a<0$ and $\Theta_\mu\ne\emptyset$, then $-1/(4a)=N\in\NN$ and   $\mu$ is a binomial ensemble $Bin(N, q_1, q_2)$ from Definition \ref{Def-Bin} for some $q_1,q_2$.
\item  If $\mu$ is a Meixner ensemble on $\HH_{2,\beta}$  with $a=0$, $b>0$ and $\Theta_\mu\ne\emptyset$, then $\mu$ is a Poisson ensemble from Definition \ref{Def-Poiss} for some $\la>0$.
\item If $\mu$ is a Meixner ensemble on $\HH_{2,\beta}$ with  $b^2>4a>0$ and $\Theta_\mu\ne\emptyset$, then $\mu$ is one of the negative binomial ensembles from Definition \ref{Def-NB}.
\item  If $\mu$ is a Meixner ensemble on $\HH_{2,\beta}$  with    $a=b=0$ and $\Theta_\mu\ne\emptyset$, then $\mu$  is one of the Gaussian ensembles from Definition \ref{Def-Gauss}.
\item If $\mu$ is a Meixner ensemble on $\HH_{2,\beta}$  with  $b^2=4a>0$ and $\Theta_\mu\ne\emptyset$, then $\mu$ is one of the gamma ensembles from Definition \ref{Def-Gamma}.
\item %
If $\mu$ is a Meixner ensemble on $\HH_{2,\beta}$  with $b^2<4a$ and $\Theta_\mu\ne\emptyset$, then either $\mu$  is the trivial ensemble of the form $\xi \mI$ or $b=0$ and $\mu$ is  the exceptional hyperbolic Meixner ensemble from Definition  \ref{P3.11}. %
\end{enumerate}

\end{theorem}
(The seventh case of non-random ensembles arises from degenerate affine transformations.)

\subsection{Laplace transforms of binomial, negative binomial and Poisson ensembles}

\begin{lemma}\label{S-Laplace}
If $\mP$ is a random projection in Bin$(1,1,0)$, then with $\sigma_j=\sigma_j(\theta)$, we have
\begin{equation}\label{F-function}
\E (e^{\langle\theta |\mP\rangle})=e^{\sigma_1/2}
\calI_{(\beta-1)/2}\left(\sqrt{\frac{\sigma_1^2}{4}-\sigma_2}\right),
\end{equation}
where $\calI_\nu$ is a Bessel function \eqref{F}.
\end{lemma}

\begin{proof}
To construct $\mP$, we choose a direction $\macierz{X_1\\X_2}$ at random by taking
two independent standard normal random variables as the components of the vector. Note that this means that  $X_1=Z_1+i Z_2$ in the case $\beta=2$ or $X_1=Z_1+iZ_2+jZ_2+kZ_4$ 
in the case $\beta=4$, where $Z_1,Z_2,Z_3,Z_4$ are the standard
real-valued independent normal random variables.
The matrix representation of $\mP$  is
$$\mP=\frac{1}{|X_1|^2+|X_2|^2}\macierz{|X_1|^2 & \bar{X}_2X_1\\
\bar{X}_1X_2& |X_2|^2}.$$
Due to invariance under rotations, without loss of generality we take diagonal $\theta=\macierz{\theta_1&0\\0&\theta_2}$, so that
$$\langle \theta |\mP\rangle=\tr(\theta \mP)=\frac{1}{|X_1|^2+|X_2|^2}\left(\theta_1 |X_1|^2+\theta_2|X_2|^2\right)=\theta_1 U+\theta_2(1-U).$$
Since $|X_1|^2, |X_2|^2$ are independent  gamma distributed with shape parameter $\beta/2$ and scale parameter $1/2$, the distribution of $U=|X_1|^2/(|X_1|^2+|X_2|^2)$ is Beta$(\beta/2,\beta/2)$.
This gives %
$$\E (e^{\langle\theta |\mP\rangle})=
\frac{\Gamma(\beta)}{\Gamma^2(\beta/2)}\int_0^1 e^{u\theta_1+(1-u)\theta_2}u^{\beta/2-1}(1-u)^{\beta/2-1} du.
$$
Substituting $u=(1+v)/2$ so that $u(1-u)=(1-v^2)/4$ %
we get  %
\begin{equation}\label{F-function-I}
L_\mP(\theta)=e^{\theta_1/2+\theta_2/2} \frac{\Gamma(\beta)}{2^{\beta-1}\Gamma^2(\beta/2)}
\int_{-1}^1 e^{(\theta_1-\theta_2)v/2}(1-v^2)^{\beta/2-1} dv.
\end{equation}
The integral is expressed in terms of the Bessel functions in \cite[\S3.71 (9)]{Watson:1944}. %
\end{proof}
\begin{remark}
Denoting by $\theta_1,\theta_2\in \RR$ the eigenvalues of $\theta$, 
$$
\E(e^{\langle\theta |\mR\rangle})=\begin{cases}I_0(\theta_1-\theta_2) & \mbox{ if } \beta=1,\\ \\
\frac{\sinh(\theta_1-\theta_2)}{\theta_1-\theta_2}& \mbox{ if } \beta=2,\\ \\
3\left(\frac{\cosh(\theta_1-\theta_2)}{(\theta_1-\theta_2)^2}-\frac{\sinh(\theta_1-\theta_2)}{(\theta_1-\theta_2)^3}\right)& \mbox{ if } \beta=4.
\end{cases}
$$

\end{remark}

 Combining this with the Laplace transforms derived in the proofs of Propositions \ref{P.B0}. \ref{P.P0}, and \ref{P.NB0}, we get.
\begin{corollary}\label{Bin-Laplace} %
\begin{enumerate}
\item If $\mX$ has Bin$(N, q_1, q_2)$ law, then its Laplace transform is well defined for all  $\theta\in\HH_\beta$, and is given by
\begin{equation}\label{Bin-function}
\E( e^{\langle \theta |\mX\rangle })=e^{N\sigma_1/2}\left( q_0e^{-\sigma_1/2} + q_1
\calI_{(\beta-1)/2}\left(\sqrt{\frac{\sigma_1^2}{4}-\sigma_2}\,\right)+ q_2e^{\sigma_1/2} \right)^N,
\end{equation}
where $ q_0=1- q_1- q_2$.
\item If $\mX$ has Poiss$(\la_1,\la_2)$ law, then its Laplace transform is well defined for all  $\theta\in\HH_\beta$, and is given by
\begin{equation}\label{Poiss-function}
\E (e^{\langle \theta |\mX\rangle })= \exp\left(\la_1e^{\sigma_1/2}
\calI_{(\beta-1)/2}\left(\sqrt{\frac{\sigma_1^2}{4}-\sigma_2}\,\right)+\la_2 e^{\sigma_1}-\la\right)\,,
\end{equation}
where $\la=\la_1+\la_2$.
\item If $\mX$ has NB$(r, q_1,q_2)$ law, then its Laplace transform is well defined for $\theta$ in a neighborhood of $\mO$ in $\HH_{2,\beta}$, and is given by
\begin{equation}\label{NB-function}
\E (e^{\langle \theta |\mX\rangle })=
\frac{(1-(q_1+q_2))^r}{\left(1-\left(q_1e^{\sigma_1/2}
\calI_{(\beta-1)/2}\left(\sqrt{\frac{\sigma_1^2}{4}-\sigma_2}\,\right) +q_2 e^{\sigma_1}\right)\right)^{r}}\;.
\end{equation}
\end{enumerate}

\end{corollary}
\begin{remark} \label{Rem4Thm5.1}
A calculation shows that for a given $b^2>4a>0$,  parameter $p=1-q_1-q_2$ in \eqref{NB-function} can range over the set
$$\{p\in(0,1): p/(2-p)\leq \kappa/b \leq p\}\,,$$ while $q_2=(b p-\kappa)/\kappa\in[0,1-p]$ and $q_1=1-p-q_2$ are uniquely determined by $p$. (Here $\kappa=\sqrt{b^2-4a}$, $b>0$.)
\end{remark}

\subsection{Solutions of the system of PDEs for $n=2$}
\label{Sect: Quadratic Solutions}

We  consider in more detail the case $n=2$ with
\begin{equation}\label{U2}
U=\{(\sigma_1,\sigma_2):  \; 4\sigma_2< \sigma_1^2\}.
\end{equation}
We  require that the solutions extend by continuity to $\sigma_1=\sigma_2=0$.

\subsubsection{Case  $a\ne0$ (Theorem \ref{T1}(i))}
The system of PDEs simplifies to:
\begin{eqnarray}
  g_{11}-\sigma_2 g_{22}&=&\frac{\beta}{2}g_2+(b^2-4a) g ,\label{I b}\\
  2g_{12}+\sigma_1g_{22}&=&0 .\label{I a}
\end{eqnarray}
We will use as the  initial conditions
\begin{equation}\label{I ini}
g(0,0)=1,\; \frac{\partial g(\sigma_1,0)}{\partial \sigma_1}\big|_{\sigma_1=0}=-b.
\end{equation}

Denote $\kappa=\sqrt{|b^2-4a|}$.

\begin{proposition}\label{P6.1} Consider   the system \eqref{I b}, \eqref{I a} with initial conditions \eqref{I ini}.
\begin{enumerate}
  \item\label{i} For   $b^2=4a>0$ all solutions are
  $$g(\sigma_1,\sigma_2)=1-b \sigma_1 +C ( {  \beta\,{{{\sigma }_1}}^2} +4\sigma_2), $$
where $C$ is an arbitrary constant. Accordingly,
\begin{equation}\label{f ans}
L(\theta)=\frac{\exp(-\sigma_1(\theta)/b)}{(1-b \sigma_1(\theta) +C ( {  \beta\,{{{\sigma }_1}^2(\theta)}} +4\sigma_2(\theta)))^{1/b^2}}\;.
\end{equation}

\item\label{ii} For $b^2>4a$
 all solutions are
  \begin{equation}\label{gggg}
  g(\sigma_1,\sigma_2)=C_1 e^{\kappa \sigma_1}+C_2 e^{-\kappa \sigma_1}+
 C_3  \calI_{(\beta-1)/2}(\kappa\sqrt{\sigma_1^2/4-\sigma_2}),
\end{equation}
where $\calI_{(\beta-1)/2}$ is defined in \eqref{F}, and $C_1,C_2,C_3$ are arbitrary real numbers such that $C_1+C_2+C_3=1$ and $C_2-C_1=b/\kappa$.
  Accordingly,
    \begin{equation}\label{Bin-answer}
 L(\theta)=e^{- b \sigma_1/(4a)}\left(C_1e^{\kappa\sigma_1}+C_2e^{-\kappa\sigma_1}+C_3 \calI_{(\beta-1)/2}\left(\kappa\sqrt{\sigma_1^2-4\sigma_2}\right)\right)^{-{1}/{(4a)}},
\end{equation}
which can also be written as
\begin{multline}\label{NB-cosh}
 L(\theta)=e^{- b \sigma_1/(4a)}\Big((1-\la)\cosh(\kappa\sigma_1)-\frac{b}{\kappa}\sinh(\kappa\sigma_1)\\+\la \calI_{(\beta-1)/2}\left(\kappa\sqrt{\sigma_1^2-4\sigma_2}\right)\Big)^{-{1}/{(4a)}},
\end{multline}
where $\la\in\RR$ is an arbitrary constant.

 \item\label{iv} For   $b^2<4a$ all solutions are

$$g(\sigma_1,\sigma_2)=(1-\lambda) \cos \kappa\sigma_1-\frac{b}{\kappa}\sin \kappa\sigma_1+\lambda \calJ_{(\beta-1)/2}\left(\kappa\sqrt{\sigma_1^2-4\sigma_2}\right),$$
where $\la$ is an arbitrary real constant.
 Accordingly,
\begin{multline}\label{H-ans}
L(\theta)=e^{-b \sigma_1/(4a)} \Big((1-\lambda) \cos \kappa\sigma_1-\frac{b}{\kappa}\sin \kappa\sigma_1
 +\lambda \calJ_{(\beta-1)/2}\big(\kappa\sqrt{\sigma_1^2-4\sigma_2}\big)\Big)^{-1/(4a)}.\\
\end{multline}
 \end{enumerate}
\end{proposition}
\begin{proof}

The solutions that depend only on  $\sigma_1$ satisfy equation $g_{11}=(b^2-4a) g$ with two linearly independent solutions:
\begin{eqnarray}
& e^{\kappa\sigma_1},\;e^{-\kappa\sigma_1}  & \mbox{ if $b^2-4a>0$} ,\label{+}\\
& \cos\left(\kappa\sigma_1\right),\;\sin\left(\kappa\sigma_1\right)&  \mbox{ if $b^2-4a<0$} \label{-},\\
&1,\; \sigma_1 & \mbox{ if $b^2-4a=0$} .\label{0}
\end{eqnarray}

To find the solutions that depend also on $\sigma_2$, we first note that  \eqref{I a} implies that
\begin{equation}\label{C-eqn}
2g_1+\sigma_1 g_2=C(\sigma_1),
\end{equation}
so
\begin{equation}\label{C'-eqtn}
2g_{11}+g_2+\sigma_1 g_{12}=C'(\sigma_1).
\end{equation}
We now  eliminate $g_{11}$ and $g_{12}$ from the equations. Subtracting from \eqref{C'-eqtn} the linear combination of $2$ times equation \eqref{I b} and $\sigma_1/2$ times equation \eqref{I a} we get
\begin{equation}\label{ODE I}
(\sigma_1^2/4 -\sigma_2) g_{22}- \frac{1+\beta }{2}g_2-(b^2-4a) g=-C'(\sigma_1)/2.
\end{equation}
For fixed $\sigma_1$, we consider \eqref{ODE I} as a differential equation with respect to $\sigma_2$ in the interval $\sigma_2<\sigma_1^2/4$, since we need $g$ in this domain, see \eqref{U2}. For this reason, we denote
$t=\sigma_1^2/4-\sigma_2$ and, assuming $b^2\ne 4a$, we denote
$$y(t)=g(\sigma_1,\sigma_2)- \frac{C'(\sigma_1)}{2 (b^2-4a)}.
$$
Thus $g_2=-y'$, $g_{22}=y''$ and \eqref{ODE I} becomes
\begin{equation}\label{Gerard-I}
2t y''+(1+\beta)y'\pm 2\kappa^2 y =0,
\end{equation}
where the sign is chosen according to the sign of $4a-b^2$.
Substitution  $y(t)=t^{(1-\beta)/4}u(x)$  with $x=2\kappa \sqrt{t}$ converts
 \eqref{Gerard-I} into the Bessel equation
 $$x^2 u''+ x u'+\left(\pm x^2-\left(\frac{\beta-1}{2}\right)^2\right)u=0.$$
Using \cite[9.6.18]{abramowitz+stegun}  or \cite[9.1.20]{abramowitz+stegun},   after some calculation one verifies that  all solutions of \eqref{Gerard-I} which are bounded in a neighborhood of $0$ are proportional to
$\calI_{(\beta-1)/2}(\kappa\sqrt{t})$ when $b^2>4a$ and to $\calJ_{(\beta-1)/2}(\kappa\sqrt{t})$ when $b^2<4a$, see \eqref{F} and \eqref{G}.
Substituting $g(\sigma_1,\sigma_2)=K(\sigma_1) \calI_{(\beta-1)/2}(\kappa\sqrt{\sigma_1^2/4-\sigma_2})$
back into \eqref{I a}   we get that \eqref{I a} holds if $K'=0$. So we deduce that  $K$ is constant, and a direct verification shows that $g$ solves also \eqref{I b}.
(We repeat the same reasoning with $\calJ_{(\beta-1)/2}$ instead of $\calI_{(\beta-1)/2}$ if $b^2<4a$.)

Similarly, one works out the solutions for $b^2=4a>0$. In this case, after taking $y(t)=g(\sigma_1,\sigma_2)+ 2\sigma_2 C'(\sigma_1)/(1+\beta)$   we get equation \eqref{Gerard-I}  with $\kappa=0$.
After some work this leads to elementary solutions proportional to $g(\sigma_1,\sigma_2)=\frac{\beta\,{{{\sigma }_1}}^2}{4} + {{\sigma
}_2}$ and to additional solutions  $g(\sigma_1,\sigma_2)=\left(\sigma_1^2-4\sigma_2\right)^{(1-\beta)/2}$ for $\beta>1$ or $g(\sigma_1,\sigma_2)=\log(\sigma_1^2-4\sigma_2)$ for $\beta=1$, which are unbounded at the origin.

 To conclude the proof of \eqref{i}, we note that
 the general solution that is defined at $\sigma_1=\sigma_2=0$ is
$$g(\sigma_1,\sigma_2)=C_1+C_2\sigma_1+C_3 ( {  \beta\,{{{\sigma }_1}}^2} +4\sigma_2).$$
The initial conditions determine $
C_1= 1,\; C_2=-b$, which gives \eqref{f ans}.

To conclude the proof of \eqref{ii}, we note
 that the general solution bounded at the origin is \eqref{gggg},
 and the initial conditions are satisfied when $C_1+C_2+C_3=1$ and $C_2-C_1=b/\kappa$. This gives \eqref{Bin-answer}.

The initial conditions similarly imply \eqref{H-ans}.

\end{proof}

We are now ready to prove that gamma and hyperbolic Meixner ensembles are indeed Meixner.
\begin{proof}[Proof of Proposition \ref{Gamma-Laplace}]
 To calculate the moments, we differentiate \eqref{Gindikin} to get $\E (\xi_0)=2cp$, $\E(\xi_1)=0$,
 $\mbox{Var}(\xi_0)=2p(2c^2-1)$, $\E (\xi_1^2)=2p$, and  $\E(\xi_{\ell_1}\xi_\ell{_2})=0$ for $\ell_1\ne \ell_2$.

For   the standardized gamma ensemble $\widetilde\mX$ with  $\E (\widetilde \mX)=\mO$ and  $\E (\widetilde \mX^2)=\mI$, from
\eqref{Gamma-function} we get
\begin{equation}\label{Gamma-std}
\E (e^{\langle \theta |\widetilde\mX\rangle })=e^{-\sqrt{p}\tr\theta} \left( 1- \tr\theta/\sqrt{p}+  D(\beta\tr\theta^2+4\det\theta)\right)^{-p},
\end{equation}
where  $D=1/(4pc^2(1+\beta))$.  This is  \eqref{f ans} with $p=1/b^2$, so property \eqref{pre-ME}  follows from the converse part of Theorem \ref{T1}.
For the standardized version, \eqref{Meixner-Anshelevich} holds with $a=1/(4p)$ and $b=1/\sqrt{p}$.
Formula for the parameters $A,B,C$ is now re-calculated from formula \eqref{ABC2ab}.

Note that the admissible ranges of parameters are $p>\beta/2$ and $0\leq D<1/(4p(1+\beta))$.

\end{proof}

\begin{proof}[Proof of Proposition \ref{H-Laplace}]
We compute the moments of $\mX$ entrywise:   differentiating \eqref{H*H*} we get   
 $\E \xi_j=0$, $\E\xi_j^2=\alpha\la/(1+\beta)$ for $j\geq 1$, and
 $\E(\xi_{\ell_1}\xi_\ell{_2})=0$ for $\ell_1\ne \ell_2$. This gives the first two moments of $\mX$.

For   the standardized hyperbolic Meixner ensemble $\widetilde\mX$ with  $\E(\widetilde \mX)=\mO$ and  $\E (\widetilde \mX^2)=\mI$, from
\eqref{H-function} we get
$$
\E e^{\langle\theta|\widetilde \mX\rangle}=
 \Big(
1+
\calJ_{(\beta-1)/2}\big(\sqrt{\tr \theta^2-4\det\theta}/\sqrt{\alpha}\big)\Big)^{-\alpha}.
$$
Substituting $1/\alpha=4a$, we get \eqref{H-ans} with $b=0$, $\kappa=1/\sqrt{\alpha}$, and $\la=1$.
So   from the converse part of Theorem \ref{T1} we see that \eqref{Meixner-Anshelevich} holds with $a=1/(4\alpha)$ and $b=0$. Formula for the parameters $A,B,C$ is now re-calculated from formula \eqref{ABC2ab}.

\end{proof}

 \subsubsection{Poisson case: $a=0,b\ne0$  (Theorem \ref{T1}(ii))}

The system of PDEs simplifies to:
\begin{eqnarray}
  g_{11}-\sigma_2 g_{22}&=&\frac{\beta}{2}g_2+ 2b g_1, \label{IIa}\\
  2g_{12}+\sigma_1g_{22}&=&2b g_2 .\label{IIb}
\end{eqnarray}
We seek solutions such that
\begin{equation}\label{II-ini}
g(0,0)=0,\;   \frac{\partial g(\sigma_1,0)}{\partial \sigma_1}\big|_{\sigma_1=0}=\frac{1}{2b}.
\end{equation}

\begin{proposition}[Poisson ensemble]
\label{Prop Poisson}Consider the system \eqref{IIa}, \eqref{IIb}  with initial condition \eqref{II-ini}.
The solution is %
$$
g(\sigma_1,\sigma_2)=\frac{1}{2 b^2}\left(  (1-C)e^{2b \sigma_1}+(2C- 1)e^{b \sigma_1}\calI_{(\beta-1)/2}\left(b\sqrt{\sigma_1^2-4\sigma_2}\right)-C\right),
$$
where $C$ is an arbitrary constant.
Thus
\begin{multline}\label{Poiss-answer}
L(\theta)=\exp\Big(-\frac{\sigma_1}{2b}\\ +
\frac{1}{2 b^2}\left(  (1-C)e^{2b \sigma_1}+(2C- 1)e^{b \sigma_1}\calI_{(\beta-1)/2}\left(b\sqrt{\sigma_1^2-4\sigma_2}\right)-C \right)\Big).
\end{multline}

\end{proposition}
\begin{proof}

The solution that depends on $\sigma_1$ only is in \cite{Laha-Lukacs60}; in our notation,
$$g(\sigma_1)=C_1+C_2 e^{2 b \sigma_1}.$$
To find the solutions that depend on both variables, we
first note that \eqref{IIb} implies
\begin{equation}\label{IIc}
2g_1+\sigma_1g_2-2ag=C(\sigma_1),
\end{equation}
which we differentiate with respect to $\sigma_1$ to get
\begin{equation}\label{IId}
2g_{11}+g_2+\sigma_1g_{12}-2ag_1=C'(\sigma_1).
\end{equation}
We now eliminate $g_{11}$, $g_{12}$ and $g_1$ from the equations.
To do so, from equation \eqref{IId} we subtract $2$ times equation \eqref{IIa}, $b$ times equation \eqref{IIc}, and
$\sigma_1/2$ times equation \eqref{IIb}. This gives
\begin{equation}\label{Gerard-P}
(\sigma_1^2/4-\sigma_2)g_{22}-\frac{1+\beta}2g_2- b^2 g=-(C'(\sigma_1)+aC(\sigma_1))/2.
\end{equation}
Noting similarity with \eqref{ODE I}, we again we consider \eqref{Gerard-P} as a differential equation with respect to $\sigma_2$ in the interval $\sigma_2<\sigma_1^2/4$. With  $b\ne 0$,   $t=\sigma_1^2/4-\sigma_2$, we take
$$
y(t)=g(\sigma_1,\sigma_2)-\frac{C'(\sigma_1)+aC(\sigma_1)}{2b^2},
$$
and we get
$$2t y''+(1+\beta)y'-2b^2 y =0.$$
Therefore, the solution bounded at the origin is   $g(\sigma_1,\sigma_2)=K(\sigma_1) \calI_{(\beta-1)/2}(b\sqrt{\sigma_1^2/4-\sigma_2})$. Substituting this back into \eqref{IIb}, we get  $K'=b K$, {\em i.e.},  $K(\sigma_1)=e^{b \sigma_1}$, and the resulting
$g(\sigma_1,\sigma_2)$ solves \eqref{IIa}, too.

Thus the general solution bounded at the origin is
$$
g(\sigma_1,\sigma_2)=C_1+C_2 e^{2 b \sigma_1}+C_3  e^{b \sigma_1} \calI_{(\beta-1)/2}(b\sqrt{\sigma_1^2/4-\sigma_2}).$$
The initial conditions are satisfied when $C_1+C_2+C_3=0$ and $2C_2+C_3=1/(2b^2)$.
 After some calculations, we get the answer.
\end{proof}

\subsubsection{Gaussian  case: $a=0,b=0$  (Theorem \ref{T1}(iii))}
The system of PDEs %
simplifies to:
\begin{eqnarray}
  g_{11}-\sigma_2 g_{22}&=&\frac{\beta}{2}g_2+ 1, \label{III b}\\
  2g_{12}+\sigma_1g_{22}&=&0 .\label{III a}
\end{eqnarray}

\begin{proposition}\label{P.G} Consider the system \eqref{III b}, \eqref{III a} with zero initial conditions. The solution is
\begin{equation}\label{Gauss-PDE-sol}
C\sigma_1^2/2-(1-C)\frac{2}{\beta}\sigma_2,
\end{equation}
where $C\in\RR$ is arbitrary. Thus
\begin{equation}\label{Gauss-answer}
L(\theta)=\exp\left(C\frac{\sigma_1^2}{2}-(1-C)\frac{2}{\beta}\sigma_2\right).
\end{equation}

\end{proposition}
Since the homogeneous system is the same as in the case $b^2-4a=0$ of Proposition \ref{P6.1}(i), we omit the details.

We recall the following facts about univariate Laplace transforms, which  belong to the folklore:
\begin{theoremA}\label{P:LNN} 
Suppose $L(t)$ is a Laplace transform of a real random variable $X$. If $L(t)$ is defined on an open interval $\Theta$, then $L$ has analytic extension to $\Theta+ i\RR$. If $z_0\in \Theta+ i\RR$ is a  zero of $L$ of order $m\in\NN$, then $L^\alpha$ cannot be a Laplace transform of a probability measure unless $m \alpha\in\NN$.
\end{theoremA}

\begin{proof}[Proof of Theorem \ref{T.U}]  %
Without loss of generality, we may assume that $\mX$ is non-degenerate, and therefore that it is standardized with mean $\mO$ and variance $\mI_n$. Therefore, \eqref{Meixner-Anshelevich} holds with some constants $a,b$, as  explained in the introduction, replacing $\mX$ by $-\mX$ if needed, without loss of generality we may assume $b\geq 0$.
 So the Laplace transform of $\mX$  must be given by one of the %
 five formulas from Proposition \ref{P6.1}(i-iii), Proposition \ref{Prop Poisson} or Proposition \ref{P.G}.
 Then Corollary \ref{Bin-Laplace}, and Definitions \ref{Def-Gamma} and \ref{P3.11} indicate the ranges of parameters where the five formulas are indeed Laplace transforms of probability measures on $\HH_{n,\beta}$. 
  To end the proof we have  to show that these functions cannot be the Laplace transforms of probability measures when  the parameters fail to satisfy the conditions listed in Corollary \ref{Bin-Laplace}, and Definitions \ref{Def-Gamma} and \ref{P3.11}.
  Since the solutions of the PDEs and our definitions of Meixner ensembles depend  on the constraints satisfied by  parameters $a,b$,  we will consider each case separately.

In the proof, we repeatedly compute $L(\theta)$ on diagonal matrices
$$\theta^\pm_t=\left[\begin{matrix}
{t}/2&0 \\
0& \pm t/2
\end{matrix}\right].
$$

\begin{enumerate}
\item Case  $b^2>4a$ with $a<0$. %
In this case, all potential Laplace transforms for the standardized  ensembles are given by formula \eqref{Bin-answer}, and \eqref{Bin-function} shows that $L(\theta)$ is indeed the Laplace transform if $a=-1/(4N)$ and $C_1,C_2,C_3\geq 0$.
It remains to show that $L(\theta)$ is not a Laplace transform if $1/(4a)\not\in\NN$ or if one of the constants %
$C_j$ is negative.

Denote $\alpha=-1/(4a)>0$.
 If $C_3=0$, then $L(\frac{1}{\kappa}\theta_t^+)=(C_1e^t+C_2e^{-t})^\alpha$, and it is a classical fact  that $\alpha\in\NN$ and $C_1,C_2$ are the binomial probabilities,   compare Proposition \ref{P-Jorg-Bern}. Therefore, through the remainder of the proof we assume $C_3>0$.
We first show that $\alpha\in\NN$; the same  reasoning shows that the J\o rgensen set of a Bernoulli ensemble  on $\HH_{2,\beta}$ is $\{1,2,\dots\}$.
We write $L(\theta_t^-)=(f(s))^\alpha$, where
$$f(s)=1-C_3+C_3 \calI_{(\beta-1)/2}(s).$$
To show that $\alpha\in\NN$, we will apply Proposition \ref{P:LNN}.
We first observe that  \eqref{F} implies
\begin{equation}\label{***}
\calI_\nu(\sqrt{z})=\sum_{n=0}^\infty a_nz^n,
\end{equation}
where
$$a_n=\frac{1}{2^{2n}n!\Gamma(n+\nu+1)}.
$$
This shows that $\calI_\nu(\sqrt{z})$  is an entire function of order
$$
\rho=\limsup_{n\to\infty}\frac{\log n}{\frac1n\log a_n}=1/2.
$$
Therefore by a  refinement of
 the little Picard's theorem \cite[Theorem 9.11]{Markushevich:1977}, the equation $\calI_\nu(\sqrt{z})=(C_3-1)/C_3$ has an infinite number of roots.

 We also observe that
 \begin{equation}\label{IV}
  \frac{d}{dz}\calI_{\nu}(\sqrt{z})=\frac14 \calI_{\nu+1}(\sqrt{z}),
\end{equation}
 see \cite[page 479]{Watson:1944}. %

 Coming back to $f(s)$, we see that the function $s\mapsto f(s)$ has an infinite set $Z_0$ of roots. We claim that $Z_0$ has at least one simple root. If not, $Z_0$ would be included in the set $Z_1$ of roots of the derivative $f'(s)$.  Now from \eqref{IV} we see that $Z_1$ is the set of roots of $\calI_{(\beta+1)/2}$ which by Lommel's Theorem,
 lies on the imaginary axis (\cite[15.25]{Watson:1944}).  %
From \cite[page 199]{Watson:1944}  we see that
 $$
 \lim_{t\to \pm \infty}\calI_{\nu}(it)=0,
 $$
 therefore $\calI_\nu(it)$ cannot be equal to $(C_3-1)/C_3$ on an infinite subset of $Z_1$.
 Therefore $f$ has at least one simple zero, and from Proposition \ref{P:LNN} we deduce that $\alpha$ must be an integer.

The final step is to show that $C_1,C_2,C_3\geq 0$.
 Since the second derivative of $\log L(\frac{1}{\kappa}\theta_t^-)=\alpha \log(1-C_3+C_3\calI_{(\beta-1)/2}(t))$ is $\alpha C_3/(\beta+1)$, we see that $C_3\geq 0$.

Next we compute the second derivative of $\frac{1}{\alpha}\log L(\frac{1}{\kappa}\theta_t^+)$, which is
$$
\left(C_2 C_3+e^{2t }C_1 C_3+4C_1 C_2e^t\right)\frac{e^t}{\left(C_2+e^t \left(e^t  C_1+C_3\right)\right)^2} .
$$
Since $\alpha\in\NN$, the Laplace transform $L(\frac{1}{\kappa}\theta_t^+)$ is well defined for all real $t$.
 The sign of the second derivative for large $t$ is determined by $C_1C_3$, showing that $C_1\geq 0$.  The sign of the second derivative as $t\to-\infty$ is determined by $C_2C_3$, showing that $C_2\geq 0$. (Recall that we consider the case $C_3\ne0$ only.)

Thus Definition \ref{Def-Bin}  indeed covers all examples of Meixner ensembles on $\HH_{2,\beta}$ with $a<0$.

\item Case $b^2>0$,  $a=0$.   %
In this case, all potential Laplace transforms for the standardized  ensembles are given by formula \eqref{Poiss-answer}. %
 From \eqref{Poiss-function}, after passing to standardized $\nX$, from \eqref{Poiss-moments} we
  get \eqref{Poiss-answer} 
with $C=(\la_1+\la_2)/(\la_1+2\la_2)$ and $b=1/\sqrt{2\la_1+4\la_2}$. Since the Laplace transform of $-\nX$ is   \eqref{Poiss-answer}  with the same $C$ but $b=-1/\sqrt{2\la_1+4\la_2}$, we see that    \eqref{Poiss-answer}
  is indeed a Laplace transform when $1/2\leq C\leq 1$, $b> 0$.

To verify that $L(\theta)$   fails to be a Laplace transform of an ensemble in all other cases, we compute
$$\frac{d^2}{dt^2}\log L\left(\theta^-_t/b\right)\Big|_{t=0}=2C-1.$$
Thus we must have $C\geq 1/2$.
 Next, we note that in order for
 $$\frac{d^2}{dt^2}\log L\left(\theta^+_t/b\right) =e^t \left(4 e^t (1-C)+2
   C-1\right)$$
to be positive for large $t$, we must have $C\leq 1$.

This shows that Definition \ref{Def-Poiss} covers all examples of Meixner ensembles on $\HH_{2,\beta}$ with $b\ne 0$, $a=0$

\item Case $b^2>4a$, $a>0$. %
In this case, all potential Laplace transforms for the standardized  ensembles are given by formula \eqref{NB-cosh}.
To see which values of $\la$ correspond to the Laplace transform  \eqref{NB-function} of the negative binomial ensemble, we take $r=1/(4a)$ and define parameters $q_1,q_2$ and $p=1-q_1-q_2$ as follows
$$
\frac{2}{p}=1-\la +b/\kappa,\; q_2=(p-1) p\la,\; q_1=1-p-q_1.
$$
From the admissible range of  $p$ given in Remark \ref{Rem4Thm5.1} we see that
\eqref{NB-cosh} is a Laplace transform of an ensemble  when $1\leq 1-\la\leq b/\kappa$.

It remains to show that no other values of $\la$ are allowed.
The second derivative of
$$4a \log L(\theta^-_t/\kappa)=-\log(1-\la+\la\calI_{(\beta-1)/2}(t) ) $$ at $t=0$ is
 $- \la/(1+\beta)$, proving that $\la\leq 0$.

 Next suppose $1-\la>b/\kappa$.  Then, with $\rho=b/(\kappa(1-\la))<1$ we see that
 $$
 (1-\la)\cosh t-b/\kappa \sinh(t)=(1-\la)(\cosh t- \rho \sinh t)
 $$
 is an increasing function of $t>0$. So  $L(\theta_t^+/\kappa)$ is well defined for all $t>0$.
 We now compute the second derivative of
 \begin{equation}\label{L-NB***}
 4a\log L(\theta^+_t/\kappa)=-\log(\la+(1-\la)\cosh(t)-b/\kappa\sinh t).
\end{equation}
 We get
$$\frac{b^2-\kappa ^2 (\lambda -1)^2+\kappa  \lambda  (\kappa
   (\lambda -1) \cosh (t)+b \sinh (t))}{(-\kappa  \lambda +\kappa
   (\lambda -1) \cosh (t)+b \sinh (t))^2}\,.
   $$
 Since $\kappa(\la-1)+b<0$, this expression fails to be positive for large $t$, contradicting that $L$ is a Laplace transform of a probability measure.  Thus  we must have
 $1-\la\leq b/\kappa$.

This shows that Definition \ref{Def-NB} covers all examples of Meixner ensembles on $\HH_{2,\beta}$ with $a>0$, $b^2>4a$.

\item Case $a=0$, $b=0$ (Proposition \ref{P.G}).  It is clear that  Definition \ref{Def-Gauss} covers all examples of Meixner ensembles on $\HH_{2,\beta}$ with Laplace transform \eqref{Gauss-answer}.

\item Case $a>0$, $b^2=4a$ (Proposition \ref{P6.1}\eqref{i}).
From \cite{Letac-Wesolowski:2008} it follows that Definition \ref{Def-Gamma} covers all examples of Meixner ensembles on $\HH_{2,\beta}$ with $a>0$, $b^2=4a$.

\item %
Case $b^2<4a$ (Proposition \ref{P6.1}\eqref{iv}). The trivial Meixner case corresponds to $\la=0$. From 
\eqref{H-function}
we see that \eqref{H-ans} is a Laplace transform of a probability measure   also when $\la=1$, $b=0$.
    
It remains to show that \eqref{H-ans} fails to be a Laplace transform in the remaining cases.

Denote $\alpha=1/(4a)>0$. The second derivative of
$$\log L(\theta^-_t/\kappa)=-\alpha\log(1-\la+\la\calJ_{(\beta-1)/2}(t) ) $$ at $t=0$ is
 $\alpha \la/(1+\beta)$, proving that we must have $\la\geq 0$.

Next, for $\la\ne 1$ let  $\phi\in(-\pi/2,\pi/2)$ be such that $\sin \phi=\rho/\sqrt{(1-\la)^2+\rho^2}$.
Then the second derivative of
\begin{multline}\label{EHM-1}
\log L(\theta_t^+/\kappa)=-\alpha\log(\la+(1-\la)(\cos t-\rho/(1-\la)\sin t))\\=
\alpha\log\left(\la+(1-\la)\frac{\cos (t+\phi)}{\cos \phi} \right)
\end{multline}
 at
 $t=-\phi$ is
 $$\alpha\frac{1-\lambda }{(\frac{1}{\cos \phi}-1)
   \lambda +1}\,.
$$
Since we already know that $\la\geq 0$, this shows that we must have $\la\leq 1$.
Next we check that if $\la>0$ and  \eqref{H-ans} is a Laplace transform then $\la=1$ and $b=0$.

\begin{claim}\label{Claim-EHM}  If $\la<1$ then $\la \leq 1/2$.

\end{claim}
To prove this, we proceed by contradiction.  Suppose $1/2<\la<1$. From \eqref{EHM-1} we see that $\log L(\theta_t^+)$ is well defined for all $t>0$.
 Then the second derivative of $\log L(\theta_t^+)$ at $t=\pi-\phi$ is  
 $$\frac{4\alpha  (1-\lambda)}{1-(1+\cos \phi) \lambda }\leq \frac{4\alpha  (1-\lambda)}{1-2 \lambda }<0,$$
  giving a contradiction.
 
 \begin{claim}
   If $\la>0$ then $\la=1$.  
 \end{claim}
 To prove this, we again proceed by contradiction.
Suppose $0<\la<1$. Then 
 $$\log L(\theta^-_t/\kappa)=\frac{1}{\left(1-\la+\la \mathcal{J}_{(\beta-1)/2}(t)\right)^\alpha}$$
From the integral representation of the Bessel function \cite[\S3.3]{Watson:1944} we read out that $|\mathcal{J}_\nu(t)|\leq \mathcal{J}_\nu(0)=1$ for all real $t$. Since  from Claim \ref{Claim-EHM} we have $\la\leq 1/2$,  we see that all real arguments $t>0$ are allowed into $\log L(\theta^-_t/\kappa)$.
 Computing the second derivative of $\log L(\theta^-_t/\kappa)$, we get
$$
\frac{\la \alpha \left(\la (\calJ_{(\beta-1)/2}'(t))^2-(1-\la)\calJ_{(\beta-1)/2}''(t) -\la  \calJ_{(\beta-1)/2}(t)\calJ_{(\beta-1)/2}''(t) \right) }{(1-\la+\la \calJ_{(\beta-1)/2}(t))^2}
$$
The sign of this expression is determined by the sign of $\la (\calJ_{(\beta-1)/2}'(\ell_\beta))^2-(1-\la)\calJ_{(\beta-1)/2} ''(t) -\la  \calJ_{(\beta-1)/2}(t)\mathcal{J} ''(t)$. When $\beta=1,2,4$, this expression cannot be positive for $\la\leq 1/2$. To see this, we use Mathematica to evaluate it at a fixed point $t\in(3,4)$, say $t=\pi$.
This contradiction proves that $\la=1$.

 It remains to show that  \eqref{H-ans} is not a Laplace transform of a probability measure when $\la=1$ and $b\ne0$.  In this case, consider
$$
g(s,t)=\log L(\frac{1}{\kappa}(\theta^+_s+\theta^-_t))=-bs/(4a)-\alpha \log\left(\calJ_{(\beta-1)/2}(t)-b \sin (s)/\kappa\right).
$$
Then $g(s,t)$ is well defined on $(s,t)\in[0,\pi)\times(-\ell_\beta,\ell_\beta)$, where
 $\ell_\beta$  is the first positive zero of $ \calJ_{(\beta-1)/2}(x)$ {\em i.e.}, $\ell_1\approx 2.40483$, $\ell_2=\pi\approx 3.14159$, and $\ell_4\approx 4.49341$.

 We calculate the Hessian
$$
H(s,t)=\det \left[\begin{matrix}
\frac{\partial ^2}{\partial s^2}g(s,t) & \frac{\partial ^2}{\partial s \partial t} g(s,t) \\
\frac{\partial ^2}{\partial s \partial t} g(s,t)& \frac{\partial ^2}{\partial t^2}g(s,t)
\end{matrix}
\right]$$
at $s=0$ and we get
$$ H(0,t)=-\frac{b^2 \alpha ^2
   \calJ_{(\beta-1)/2}''(t)}{\kappa ^2
  ( \calJ_{(\beta-1)/2}(t))^3}.$$
Using  Mathematica, we verify that $ \calJ_{(\beta-1)/2}''(0) \calJ_{(\beta-1)/2}''(\ell_\beta)<0$, so for $b\ne 0$ the Hessian must change sign over the domain $(s,t)\in [0,\pi)\times(-\ell_\beta,\ell_\beta)$. Thus $\log L(\theta)$ cannot be a convex function, {\em i.e.},
when $\la=1$, $b\ne0$, function $L(\theta)$ cannot be  a Laplace transform of a probability measure for any $\alpha\ne 0$.

This shows that the only non-trivial Meixner ensemble $\HH_{2,\beta}$ is the exceptional hyperbolic ensemble from Definition \ref{P3.11}.

\end{enumerate}

\end{proof}

\section{Additional observations}\label{Sect_Add_Obs}
\subsection{Independence of $\mS$ and $\mS^{-1}\mX^2\mS^{-1}$}\label{Sect: Independence}
Bo\.zejko and Bryc \cite[Remark 5.8]{Bozejko-Bryc-04} raise the
question whether there exists a non-trivial law on positive
$n\times n$ matrices such that \begin{equation}   \label{Z,S}
\mZ:=\mS^{-1}\mX^2\mS^{-1}, \; \mS:=\mX+\mY
\end{equation} are independent when $\mX$ and $\mY$ are  i.i.d. matrices
with this law. Such laws could provide matrix models for the "free gamma" law.

Here we answer this question in negative, at least for  the laws on positive $2\times 2$ matrices which are invariant under orthogonal/unitary/symplectic group.
We show that in the case of $2\times 2$ positive random matrices,
$\mS^{-1}\mX^2\mS^{-1}$ and $\mS$ are independent only if  $\mX$ arises as a gamma random variable multiplied by $\mI_2$.

\begin{proposition}\label{C2.3} If $\mX,\mY\in\HH_{2,\beta}$ are  i.i.d.  square-integrable non-degenerate positive
random matrices  such that $\mS=\mX+\mY$ and  $\mS^{-1}\mX^2\mS^{-1}$ are
independent and rotation invariant, then $\mX=\xi \mI_2$ and $\xi$ has a univariate gamma law.
\end{proposition}
\begin{lemma}\label{P2.3} Suppose $\mX,\mY\in\HH_{2,\beta}$ are positive  i.i.d.  random
matrices with the Laplace transform
\begin{equation}\label{L-quadr}
L(\theta)=%
(1+A\tr \theta+B\tr^2\theta+C\tr\theta^2)^{-p}.
\end{equation}

Let $\mS=\mX+\mY$. If real random variables $\det(\mX)/\det(\mS)$ and $\det \mS$
are independent, then one of the following cases occurs:
\begin{itemize}
\item[(i)]
 $A=B=C=0$. (Then $\mX=\mI$.)
\item[(ii)] $A=\pm 2 \sqrt{B}$, $C=0$. (Then $\mX=\frac{A}{2} \xi \mI$ where $\xi$ is univariate gamma with shape parameter $2p$.)
\item[(iii)] $A=\pm \sqrt{2B}$, $C=-B$. (Then
$\mX=A\mY$ where $\mY$ has a Wishart law with shape parameter $p$, {\em i.e.}, with the Laplace transform $L(\theta)=(1-  \tr(\theta) + \det \theta)^{-p}=\det(\mI-\theta)^{-p}$.)

\end{itemize}
\end{lemma}
\begin{proof}%
The independence of determinants implies
$$
\E\left( \det \mX \exp \tr(\theta \mS)\right)= c_0 \E \left( \det \mS \exp \tr(\theta \mS)\right) ,
$$
where $c_0=\E \left(\frac{\det \mX}{\det \mS}\right) >0$.
With $\theta=\left[\begin{matrix}\theta_1&\theta_{12}\\\theta_{12}&\theta_2
\end{matrix}\right]$, this becomes
$$
L(\theta) \left(\frac{\partial^2}{\partial \theta_1 \partial \theta_2}-\frac{\partial^2}{4\partial \theta_{12}^2}\right)L(\theta)=c_0 \left(\frac{\partial^2}{\partial \theta_1 \partial \theta_2}-\frac{\partial^2}{4\partial
\theta_{12}^2}\right)L^2(\theta).
$$
Writing  \eqref{L-quadr} as $L(\theta)=1/(f(\theta))^{p}$,  when $c_0\ne  \frac{p+1}{2(2p+1)}$ we get
$$
\frac{\partial f}{\partial \theta_1 }\frac{\partial f}{\partial \theta_2}-\frac1{4}\left(\frac{\partial f }{\partial \theta_{12}}\right)^2\ =c_1 f(\theta)\left(\frac{\partial^2 f}{\partial \theta_1 \partial \theta_2}-\frac14\frac{\partial^2
f}{\partial \theta_{12}^2}\right).$$
where
$$
c_1=\frac{2c_0-1}{2(2p+1)c_0-p-1}\;.
$$
Solving the resulting equations for $A,B,C,c_1$ we verify that the only non-zero solutions correspond either  to
$c_1=2$, $A=\pm 2\sqrt{B}$,  $C=0$ with $f(\theta)=(1+A \tr(\theta)/2)^2$, or to
$c_1=2/3$ with $A=\pm \sqrt{2B}$, $C=-B$.

When  $c_0= \frac{p+1}{2(2p+1)}$, the equations give $C=2B$ so
$f(\theta)= 1+A\tr\theta+B(\tr^2\theta+2\tr\theta^2)$.
Using  \cite[Theorem 3.1]{Letac-Wesolowski:2008} one can verify that in this case $L(\theta)$ is not  a Laplace transform.
\end{proof}
\begin{proof}[Proof of Proposition \ref{C2.3}]
Denote $\mZ=\mS^{-1}\mX^2\mS^{-1}$. Note that this is a positive-define matrix.
First we note that for positive matrices the Laplace transform is defined on $\theta\in\HH$ when
$-\theta$ is positive.

 Next we note that rotation invariance of the law of $\mX$ implies that $\E (\mX)=\alpha\mI$ with $\alpha>0$,  %
 $\E (\mX^2)=(\alpha^2+\beta^2)\mI$ with $\beta^2>0$, and
 $\E (\mZ)=c_0\mI$. Therefore,
$$\E (\mX^2|\mS)=\E (\mS \mZ \mS|\mS)=c_0 \mS^2.
$$
Taking the expected value of both sides, $c_0=\frac{\beta^2}{2(\beta^2+2\alpha^2)}$.
We also note that by exchangeability, $\E (\mX^2|\mS)=\E (\mY^2|\mS)$.
So
$$\E \big((\mX-\mY)^2\big|\mS\big)= \E (2\mX^2+2\mY^2-\mS^2|\mS) =(4c_0-1)\mS^2.$$

Passing to standardized matrices $\mX'=(\mX-\alpha\mI)/\beta$, we see that \eqref{Meixner-Anshelevich} holds with $b^2=4a=\alpha^2/\beta^2$.
Thus by Proposition \ref{P6.1}(i),
$$\E \left(\exp \tr(\theta\mX)\right)=e^{\alpha\tr\theta}\E \left(\exp \tr(\beta \theta\mX')\right) $$ and
equation \eqref{L-quadr} holds with $p=\beta^2/\alpha^2>1$.

Since $\det(\mX)/\det(\mS)=\sqrt{\det \mZ}$ is independent of $\det \mS$,   by Lemma \ref{P2.3}, there are only two non-degenerate  choices for the law of $\mX$.
It is known that  \eqref{Meixner-Anshelevich} does not hold for the Wishart law, see \cite[pg. 582]{Letac-Massam-98}, which
excludes the law from Lemma \ref{P2.3}(iii). Thus $\mX$ is the identity matrix multiplied by a Gamma random variable.
\end{proof}

\subsection{Some series solutions of the general system of PDEs}\label{Sect:Series_Sols}

 The Laplace transform of  Bernoulli ensembles can be computed readily. For example, if  $\mP$ is the orthogonal projection onto a randomly rotated line through the origin in $\CC^3$ then from \cite{Letac:2007} we have
 for %
  $\theta=Diag[\theta_1,\theta_2,\theta_3]$,
\begin{equation}\label{P-theta}
\E\exp\langle\theta|\mP\rangle=2\frac{e^{\theta _1}-e^{\theta  _2}}{\left(\theta  _1-\theta
_2\right) \left(\theta  _2-\theta
  _3\right)}+2\frac{e^{\theta  _1}-e^{\theta  _3}}{\left(\theta
_1-\theta  _3\right) \left(\theta
  _3-\theta  _2\right)}.
\end{equation}
However, it does not seem to be obvious how to express this Laplace transform in terms of the elementary symmetric functions
$\sigma_1=\theta_1+\theta_2+\theta_3, \; \sigma_2=\theta_1\theta_2+\theta_1\theta_3+\theta_2\theta_3, \;\sigma_3=\theta_1\theta_2\theta_3$. The following result gives the expansion for rank 1 projections in $\HH_{n,\beta}$ for all $n\geq 2$.    When combined  with properties of Meixner ensembles from  Section \ref{Sect:Meixner_Ensemble}, this can be used to produce  some  solutions of the system of PDEs, see e.g., Corollary \ref{C-P-P1}.

 Recall the Pochhammer symbol
 \begin{equation}\label{Pochhamer}
 (b)_k=\frac{\Gamma(b+k)}{\Gamma(b)}=b(b+1)\dots(b+k-1)\,.
\end{equation}

\begin{proposition}\label{P-P1} If $\mP_1$ is an random projection of ${\HH_{n,\beta}}$  invariant by rotation with trace 1, then $\E e^{\langle\theta|\mP_1\rangle}=L_n(\theta)$ has series expansion
$$
L_n(\theta)=\sum_{k=0}^\infty
\frac{(-1)^k}{\left(\frac{n\beta}{2}\right)_k}\sum_{\nu_1+2\nu_2+3\nu_3+\dots=k}
\frac{(-1)^{\nu_1+\nu_2+\dots}\left(\frac{n\beta}{2}\right)_{\nu_1+\nu_2+\dots}}{\nu_1!\nu_2!\nu_3!\dots} \sigma_1^{\nu_1}(\theta)\sigma_2^{\nu_2}(\theta)\dots
$$
\end{proposition}

 Since $\mP_2=\mI_3-\mP_1$ for $n=3$, we have the following.
 \begin{corollary}\label{C-P-P1}
 The Laplace transform of the Bernoulli ensemble on $\HH_{3,\beta}$ with parameters $q_1,q_2,q_3$ is
  $$
L(\theta)=q_0+q_1L_3(\theta)
+q_2e^{\sigma_1}L_3(-\theta)+q_3 e^{\sigma_1}\,.
$$

\end{corollary}

For the proof of the proposition we use facts about Jack polynomials (spherical polynomials) which require additional notation.
 \subsubsection{Partitions}
 A partition $\kappa=(k_1,k_2,\dots,k_r)$ of $k$ into at most $r$ parts is the sequence
 $k_1\geq k_2\geq \dots \geq k_r\geq 0$ such that $k_1+k_2+\dots+k_r=k$. We write $\kappa\vdash k$.  Other standard notation is $|\kappa|=k_1+k_2+\dots$ (which is just $k$). We denote the length (the number of parts) of a partition $\kappa$ by $\ell(\kappa)=\#\{j: k_j>0\}$.

 \subsubsection{Pochhammer symbols}
For a partition $\la\vdash k$, we  need the $\beta$-Pochhammer symbol:
 \begin{equation}\label{Pochhamer2}
 (b)_{\lambda;\beta}=\prod_{j=1}^{\ell(\la)}\left(b-\frac{\beta}{2}(j-1)\right)_{\la_j}\,.
\end{equation}
\subsubsection{Jack polynomials}
 For a not necessarily symmetric  $n\times n$ matrix $\theta$ with eigenvalues $
 \theta_1,\theta_2,\dots,\theta_n$, the Jack polynomials $\{\JJJJ_\kappa(\theta;\beta)\}$ are the symmetric polynomials in $
 \theta_1,\theta_2,\dots,\theta_n$, indexed by the partitions $\kappa$ of $k=|\kappa|$ into  $r=\ell(\kappa)\leq n$ parts.   In  \cite[formula (1)]{Jack:2006a} the Jack polynomial $\JJJJ_\kappa(\theta;\beta)$ is defined
 as  the coefficient at $t_1t_2\dots t_r$ in the expansion of $$
(2/\beta)^r\det(I_n-t_{1} \theta^{k_1}-t_{2}\theta^{k_2}-\dots-t_r\theta^{k_r})^{-\beta/2} \,. $$
 It is convenient to set $ \JJJJ_\kappa(\theta;\beta)=0$   if the partition $\kappa$ has more than $n$ parts.
 (The usual notation for $ \JJJJ_\kappa(x;\beta)$ is $C_\kappa(x;\alpha)$ with $\alpha=2/\beta$.)

We need several properties of Jack polynomials.
\begin{theoremA}
\begin{enumerate}
\item If $\ell(\kappa)\leq n$ and $A,B$ are $n\times n$ matrices, then
\begin{equation}\label{J-int}
\int_{\mathcal{K}_{n,\beta}}  \JJJJ_{\kappa}(UAU^*B;\beta)dU=\frac{ \JJJJ_{\kappa}(A;\beta) \JJJJ_{\kappa}(B;\beta)}{ \JJJJ_{\kappa}(\mI_n;\beta)}\,,
\end{equation}
where the integral is over the normalized Haar measure on the orthogonal group $O_n$ ($\beta=1$), on the unitary group $U_n$ ($\beta=2$), or on the symplectic  group $K_n$  ($\beta=4$).
\item  If $A$ is an $n\times n$ matrix and $k\geq 0$, then
\begin{equation}\label{J-tr}
(\Re \tr(A))^k=\sum_{\kappa\vdash k} \JJJJ_\kappa(A;\beta).
\end{equation}

\item If $\ell(\la)\leq \min\{m,n\}$, then
\begin{equation}\label{J(I)}
\frac{\JJJJ_{\la}(\mI_{m,n};\beta)}{ \JJJJ_{\la}(\mI_n;\beta)}=\frac{\left(\frac{m\beta}{2}\right)_{\la;\beta}}{\left(\frac{n\beta}{2}\right)_{\la;\beta}}\,.
\end{equation}
(If $\ell(\la)>m$ then $ \JJJJ_{\la}(\mI_{m,n};\beta)=0$.)
\item For a one-part partition $\la=(k)\vdash k$, we have
\begin{equation}\label{J2e}
 \JJJJ_{(k)}(A;\beta)=  \frac{(-1)^k k!}{\left(\frac{\beta}{2}\right)_k}\sum_{\nu_1+2\nu_2+3\nu_3+\dots=k}\frac{(-1)^{\nu_1+\nu_2+\dots}(\frac{\beta}{2})_{\nu_1+\nu_2+\dots}}{\nu_1!\nu_2!\dots}\sigma_1^{\nu_1}\sigma_2^{\nu_2}\dots
\end{equation}
\end{enumerate}

\end{theoremA}

Formula \eqref{J-int} appears  in \cite{Hanlon-Stanley-Stembridge-92}, where it is attributed to  \cite[(45)]{Godement:1952}, see also \cite[Corollary XI.3.2]{Faraut-Koranyi-94}.   %
 Formula \eqref{J-tr} appears in \cite[Proposition 2.3]{Stanley:1989}, see also \cite[page 234]{Faraut-Koranyi-94}.
Ref. \cite[(21)]{James:1964} gives $ \JJJJ_\kappa(\mI_m;1)$.
 Formula \eqref{J2e}  is \cite[Proposition 2.2(c)]{Stanley:1989}.

We are now ready to prove Proposition \ref{P-P1}.
\begin{proof}[Proof of Proposition \ref{P-P1}]Recall that  $\mP_1=\mU^*\mI_{1,n} \mU$.
From \eqref{J-tr},
$$
\E\exp\langle\theta|\mP_1\rangle=\sum_{k=0}^\infty \frac{1}{k!} \E \tr^k(\mU\theta \mU^*\mI_{1,n})=
\sum_{k=0}^\infty \frac{1}{k!} \sum_{\kappa\vdash k} \E  \JJJJ_\kappa(\mU\theta \mU^*\mI_{1,n};\beta)
\,.$$
By \eqref{J-int}, this becomes
$$
\sum_{k=0}^\infty \frac{1}{k!} \sum_{\kappa\vdash k}  \frac{ \JJJJ_\kappa(\theta ;\beta) \JJJJ_\kappa(\mI_{1,n};\beta)}{ \JJJJ_\kappa( \mI_n;\beta)}\,.
$$
Since $ \JJJJ_\kappa(\mI_{1,n};\beta)=0$ when $\ell(\kappa)>1$, this gives
$$
\E\exp\langle\theta|\mP_1\rangle=\sum_{k=0}^\infty \frac{1}{k!}  \frac{ \JJJJ_{(k)}(\theta ;\beta) \JJJJ_{(k)}(\mI_{1,n};\beta)}{ \JJJJ_{(k)}( \mI_n;\beta)}\,.
$$
Now  by \eqref{J(I)} this becomes
$$
\E\exp\langle\theta|\mP_1\rangle=\sum_{k=0}^\infty \frac{\left(\frac{\beta}{2}\right)_{k}}{k! \left(\frac{n\beta}{2}\right)_{k}}    \JJJJ_{(k)}(\theta ;\beta),
$$
and formula \eqref{J2e} gives the answer.

\end{proof}
\subsection{Connection with free probability}\label{Sect:CWFP}

The free probability analog of Meixner family was introduced by Anshelevich, Saitoh and Yoshida
\cite{Anshelevich01,Saitoh-Yoshida01}  as the orthogonality measure of certain orthogonal polynomials. These free
Meixner laws include the Kesten law (a special case of the free binomial law, which is the free additive
convolution of the Bernoulli law), the Marchenko-Pastur law (also called the  free-Poisson law), and the  Wigner
semi-circle law (a free probability  analog of the normal law). The free Meixner family of laws again appeared as
the laws characterized by a quadratic regression property in free probability \cite{Bozejko-Bryc-04}, in a class
of Markov processes, and as generating measures of Cauchy-Stieltjes kernel families with quadratic variance
function \cite{Bryc-06-08}. See also \cite{Anshelevich:2007},  \cite{Anshelevich:2008} and \cite[Theorem
4.3]{Bryc-Wesolowski-03}.

Except for the free binomial law, the remaining  free-Meixner laws are infinitely divisible with respect to free
additive convolution and appear as limit laws of large dimensional random matrices
\cite{Benaych-Georges04,Cabanal-Duvillard:2005}. However, this correspondence is based on  the Bercovici-Pata
bijection  \cite{Berkovici-Pata99}, which, as observed by Anshelevich \cite[page 241]{Anshelevich01}, %
is different
from the correspondence based on the kernel families or regression properties.

\subsection{Connection with Euclidean Jordan algebras}\label{Sect7} The natural framework of the Anshelevich
question is rather the one of the Euclidean Jordan algebras well described in the book of Faraut and Koranyi
\cite{Faraut-Koranyi-94}. The spaces $\HH_{n,\beta}$ for $\beta=1,2,4$ are the 3 more important types of
(irreducible) Euclidean Jordan algebras. But there are two more types to be considered: the exceptional 27
dimensional Albert algebra which can be roughly be considered as a space of (3,3) Hermitian matrices  on
octonions, and the Lorentz algebra on $\RR^2\times E$ where $E$ is a Euclidean space of dimension $d-2$. In this
algebra, the product of $(x_1,x_2,\vec{v})$ and $(y_1,y_2,\vec{w})$ is
$$(x_1y_1+\<\vec{v},\vec{w}\>,x_2y_2+\<\vec{v},\vec{w}\>,
\frac{1}{2}(y_1+y_2)\vec{v}+\frac{1}{2}(x_1+x_2)\vec{w})$$
which leads to the square of $(x_1,x_2,\vec{v})$ as $(x_1^2+\|v\|^2,x_1^2+\|v\|^2, 2(x_1+x_2)\vec{v})$. A good way
to memorize this product is to write formally $(x_1,x_2,\vec{v})$ as a $2\times 2$ matrix
$X=\left[\begin{array}{cc}x_1&\vec{v}\\^t \vec{v}&x_2\end{array}\right]$, doing the same with $Y$ for
$(y_1,y_2,\vec{w})$ and to consider the Jordan
product
$$X \cdot Y=\frac{1}{2}(XY+YX).$$
If $d=\beta+2$ with $\beta=1,2,4$  this algebra is $\HH_{2,\beta}.$
For a Jordan algebra $V,$ rotational invariance in $\HH_{2,\beta}$ is generalized to the invariance by a certain
compact group $K$ acting on $V$, as  described in Faraut Koranyi \cite[page 6 and 55]{Faraut-Koranyi-94}. The
Anshelevich problem can be solved for the Lorentz algebra for any $d\geq 1$ since the rank is 2, and the preceding
study for $\HH_{2,\beta}$ extends easily to this case. Since the consideration of the Lorentz and Albert algebras
can be seen as rather academic, we have refrained from placing %
our paper in this framework.

\subsection{Unresolved questions}

\begin{question} For fixed $A,B,C$ the research of the set $M(A,B,C)$ of the ensembles on $\mathbb{H}_{n,\beta}$ which satisfies $\mathbb{E}((\textbf{X}-\textbf{Y})^2|\textbf{S})=A\textbf{S}^2+B\textbf{S}+C\textbf{I}_n$ led for $n=1$ to the Laha-Lukacs study \cite{Laha-Lukacs60} which shows $M(A,B,C)$ is a natural exponential family. For $n\geq 2$ the structure of the $M(A,B,C)$ is not really understood: is it possible to pass easily from the knowledge of one element of $M(A,B,C)$ to a knowledge of the whole set as it is the case for an exponential family? One-parameter solutions are given in Remark \ref{Rem:Exp-Fam}. The answer is hidden in the structure of the PDE system.
\end{question}
\begin{question} The dimension of the set of solutions of the PDE systems from Theorem \ref{T1} is equal $4$ if $n=2$ and is unknown if $n\geq 3$.  The subset of solutions that are analytic at zero has dimension $3$ for $n=2$, and is unknown for $n\geq 3$.
A related system of PDEs in \cite{James:1955}, has a one-dimensional set of  solutions analytic at $0$.

\end{question}
\begin{question}
Since the free binomial law is well defined for all {\em real} values of $N\geq 1$, it would be interesting to determine the J\o rgensen set of the Bernoulli ensembles for $n\geq 3$.
\end{question}
\begin{question}For  $n\geq 3$, it is not known whether a Meixner ensemble with parameters $A=-1/(2N-1)$, $B=2N/(2N-1)$, $C=0$ is a binomial ensemble from Definition \ref{Def-Bin}.
Similarly, the converses to Propositions \ref{P.P0} and \ref{P.NB0} are not known  beyond dimension $n=2$ from Theorem \ref{T.U}.
\end{question}

\begin{question} For $n\geq 3$, do we have non-Gaussian ensembles on $\HH_{n,\beta}$ such that if $\mX,\mY$ are i.i.d. then
$\E((\mX-\mY)^2|\mX+\mY)=2\mI_n$?
\end{question}

\begin{question}The distribution of eigenvalues as the dimension $n\to\infty$ was studied in \cite{Cabanal-Duvillard:2005} for a special case
of Poisson ensembles. It is interesting to study this limit for other Meixner ensembles.
\end{question}

\subsection*{Acknowledgements}
The authors benefited from discussions with Michael Anshelevich, Philippe Biane,  Marek Bo\.zejko, Alban Quadrat,
and  Marc Yor. This research was  partially supported by NSF grant  DMS-0904720  and by Taft Research Seminar
2008-09.
The first author thanks  the Universit\'e Paul Sabatier and  the second author thanks the University of Cincinnati for  hospitality during the preparation of the paper.
The authors thank the referee for a thorough and helpful review.

\end{document}